\PassOptionsToPackage{unicode}{hyperref}
\PassOptionsToPackage{naturalnames}{hyperref}

\documentclass[a4paper, 10pt]{article}
\usepackage[left=2cm, right=2cm]{geometry}
\usepackage{amsmath}
\usepackage{amsfonts}
\usepackage{amsthm}
\usepackage{mathtools}
\usepackage{comment}
\usepackage[title]{appendix}
\usepackage{enumitem}

\newtheorem{thm}{Theorem}[section]
\newtheorem{assumption}[thm]{Assumptions}
\theoremstyle{plain}
\newtheorem{lemma}[thm]{Lemma}
\newtheorem{cor}[thm]{Corollary}

\newtheorem{defn}[thm]{Definition}

\newenvironment{pfof}
    {   \par\noindent
        {\bf Proof of}
    }{
        \hfill$\Box$}
        
\usepackage{pgf, tikz, pgfplots}
\usepackage{mathrsfs}
\usetikzlibrary{arrows}
\usepackage{psfrag}

\usepackage{hyperref}
\hypersetup{
    colorlinks=true,
    linkcolor=blue,
    filecolor=magenta,      
    urlcolor=red,
}
\renewcommand{\P}{\mathbb{P}}
\newcommand{\Qn}{Q^{(N)}}

\DeclarePairedDelimiter\floor{\lfloor}{\rfloor}

\author{James R. Cruise\thanks{Riverlane Research, 59 St Andrews Street, Cambridge, CB2 3BZ. Email \texttt{james.cruise@riverlane.io}}, Fraser Daly\thanks{Department of Actuarial Mathematics and Statistics and the Maxwell Institute for Mathematical Sciences, Heriot-Watt University, Edinburgh, EH14 4AS UK. Email \texttt{f.daly@hw.ac.uk}}, Bemsibom Toh\thanks{Department of Actuarial Mathematics and Statistics and the Maxwell Institute for Mathematical Sciences, Heriot-Watt University, Edinburgh, EH14 4AS UK. Email \texttt{bct1@hw.ac.uk}}}
\title{Sample path large deviations for marked point processes in the many sources asymptotic with small buffers: Heavily and lightly loaded systems}

\begin{document}
\maketitle

\noindent{\bf Abstract}
Consider a queueing system fed by traffic from $N$ independent and identically distributed marked point processes. We establish several novel sample path large deviations results in the scaled uniform topology for such a system with a small buffer. This includes both the heavily loaded case (the load grows as $N\rightarrow\infty$) and the previously unexplored lightly loaded case (the load vanishes as $N\rightarrow\infty$); this latter case requires the introduction of novel speed scalings for such queueing systems. Alongside these sample path large deviations results, we introduce a new framework to explore the range of scalings in the many sources asymptotic for these systems.   

\vspace{12pt}


\vspace{12pt}


\section{Introduction}

In this paper we prove a number of novel and insightful sample path large deviations (SPLD) results related to the many sources asymptotic regime for queueing networks when fed by traffic from marked point processes.  These results extend previously published results for queueing systems with small buffers, i.e., where the size of the buffer grows sub-linearly in the number of sources \cite{Cruise2009a}, and a fixed load, to both heavily loaded systems (the load tends to 1 as the number of sources grows) and lightly loaded systems (the load tends to 0 as the number of sources grows).  Throughout this work, for small buffers we retain the flavour of the previous Poisson convergence result \cite{Cruise2009a} but for heavily loaded cases this is mixed with a Gaussian limit, so the system behaves as if fed by traffic from a Brownian source. For lightly loaded systems the traditional large deviations scaling breaks down, so we need to introduce a novel speed scaling for the SPLD.  

To better understand how these results relate to each other, we introduce a novel framework to parametrise the different scalings and allow for key insights into how the different scaling results relate to each other, and, furthermore, important insights as to how to select the scaling of importance when considering practical applications.  

In the many sources asymptotic \cite{Botvich1995} we are interested in a series of queueing systems where the traffic arriving externally into the $N^{\text{th}}$ system is produced by a set of $N$ independent and identically distributed traffic sources each with mean arrival rate $\lambda$.  This asymptotic was initially introduced by \cite{Weiss1986} and has been well studied in application areas including core network routers, admission control, wireless networks \cite{doi:10.1080/15326340500481762,Subramanian2011,FERNANDEZVEIGA20031376}. The particular focus on the small buffer regime is driven by a desire to understand the effect of reducing buffer sizes on core networks when modelling a range of different congestion control protocols, for example TCP, but is also relevant for data centres wanting to make use of optical networking technologies where buffering is difficult.  

In addition, the study of the small buffer regime is useful in better understanding the power which can be obtained through multiplexing, rather than using a single server for each source: if services are aggregated together, how does a smaller than linear growth in buffer size affect system performance?  

The many-sources scaling was introduced by Alan Weiss \cite{Weiss1986} to explore the quality of service in data networks. In the original scaling, both the buffer size and the service rate grow linearly with the number of sources feeding the system. In the $N^{\text{th}}$ system, the buffer size of interest is $NB$ and the service rate is of the form $N\lambda + NC$, where $\lambda>0$, $C>0$. In that work, the traffic processes were limited to Markov on/off sources, which produce traffic at a constant rate when on, and for which transitions between states are governed by a Markov chain in continuous time. Since then, a large body of work has developed around using this asymptotic framework to analyse different queuing systems. Buffet \cite{Buffet1994} and  Buffet and Duffield \cite{Duffield1994} used martingale methods to obtain bounds for the queue length in queues with Markovian arrivals within this framework. This was extended to more general sources by Botvich and Duffield \cite{Botvich1995}, who obtained a rate function for the workload process in a single server queue in both discrete and continuous time. Simonian and Guibert \cite{Simonian1995} obtained bounds on the overflow probabilities in continuous time for queues fed by on/off sources. Courcoubetis and Weber \cite{Courcoubetis1996} considered a discrete-time analogue of this system, for which they characterized the rate function associated with the overflow probability. Likhanov and Mazumdar \cite{Likhanov1999} extended the work of Courcoubetis and Weber \cite{Courcoubetis1996} by obtaining exact bounds for buffer overflow probabilities for queues with finite capacity buffers. Since then, different authors have used this framework to investigate different problems in queuing theory, including buffer overflow in multiqueue systems \cite{Shroff2004, Subramanian2011, Zhao2009, Delas2002}, conditional delays in queues \cite{Yang2006, Yang2012, Yang2006a, Shakkottai2001} and information loss across networks \cite{Bhadra2010}. Different kinds of arrival processes have also been considered, for example, continuous-time Gaussian processes with stationary increments, including fractional Brownian motion and integrated Gaussian processes \cite{Debicki2003} and on/off flows with heavy-tailed (regularly varying) on periods \cite{Zwart2004}.

There have been several attempts to move beyond the linear scaling on buffer size, and consider systems with small buffers \cite{Mandjes2001, Ozturk2004}, and very large buffers \cite{Mandjes2000}. This is usually carried out by examining a second limit in $B$ on the associated rate function for the large deviations principle, either to let $B$ increase to $\infty$ or decrease to zero. This approach does not provide the richness in scalings that we consider, and also obscures the time-scales upon which the most likely events occur. In addition this approach does not enable us to understand the joint effect of varying load and buffer size together.

The moderate deviations framework for queues with many sources was introduced by Wischik \cite{moddev}, to study heavily loaded systems. The inspiration for this work was an attempt to marry together ideas from functional central limit theorems and large deviations results for the many flows asymptotic. To this end, the scaling considered was such that in the system with $N$ sources the buffer is of size $N^{(1+\gamma)/2}B$ and the service rate is $N\lambda + N^{(1+\gamma)/2}C$ for $\gamma \in (0,1)$.  Moderate deviations have also been studied by Puhalskii \cite{Puhalskii1999}, who focused on scaling the load of the system instead of the number of sources, and obtained logarithmic asymptotes for queue length and waiting time processes in single server queues and open queuing networks in heavy traffic, and by Chang et al$.$ \cite{Chang1996}, who obtained results for queues with long-range dependent input.

Another relevant scaling in the literature is that of the small buffer scaling introduced by Cao and Ramanan \cite{CaoRamanan2002}. Here the buffer size and load are kept constant as the number of flows is increased. This paper observes the short time-scales at which events occur and makes use of this idea to show that for general point processes the behaviour of queue length is as if it had been fed by Poisson traffic. The use of weak convergence enables the authors to obtain the full distribution in the limit but does not easily allow the extension to general networks and sample path results. In comparison to this, the small buffer results proved by Cruise \cite{Cruise2009a} enable the discussion of networks and other service disciplines, but provide only tail asymptotics, since they are obtained using large deviations techniques. The small buffer scaling has been investigated in further detail in \cite{Enachescu2006, Raina2005, Gu2007, Vishwanath2009, Eun2008}.

The aforementioned results focused on the scaling of a single server queue. Wischik extended these results to consider sample path results in discrete time for single server queues \cite{Wischik2001}, and for switches \cite{Wischik1999} operating in discrete time under the many-sources asymptotic. Subramanian \cite{Subramanian2011} also used this approach in analysing multi-queue systems operating in discrete time under a Max-Weight scheduling algorithm. This sample path approach is more powerful than simply considering the behaviour of the system in one dimension, because it allows us to use tools like the contraction principle \cite[Theorem 4.2.1]{Dembo1998} to make very general statements about how various quantities of interest (such as waiting times) scale as the system scales, without needing to re-do analogous calculations from scratch.

The main contribution of this paper is the development of a number of important sample path large deviations in the scaled uniform topology for generic point process arrivals for the many sources asymptotic with small buffers. Previous studies of systems with small buffers have focused on situations where the load was constant as the number of sources increased. Here we extend this to the important situation of heavy traffic and the novel setting of very lightly loaded systems. These situations are important in providing insight into the behaviour of real systems \cite{5409854,Cao2003} but also in providing a better understanding of resource pooling and the design of future systems.  In all cases we demonstrate the importance of understanding the most likely time-scale for events of interest, which for the small buffers considered in this paper are short. These short time-scales lead to parsimony in the results, with the heavily loaded systems demonstrating a Brownian behaviour for a large class of arrival processes; similarly for lightly loaded systems, the result depends only on a small number of parameters of the system. 

In addition, to our knowledge this is the first paper to properly explore the lightly loaded case, i.e., the situation when the load tends to zero as the number of sources increases. Here we demonstrate a large deviations principle which has an unusual rate: rather than being polynomial in $N$ we have $\log(N)$. This scaling provides both theoretical insights into the richness of the different scalings which can be achieved, while also providing qualitative insight into the key features which govern the behaviour of real systems.

The final contribution in this work is the development of a novel framework for exploring a range of scalings for the many sources asymptotic. As part of the introduction of this novel framework, we explore the associated sample path scalings. These scalings provide an insight into the key features which will affect the results obtainable in each case, and also into how the various scalings relate to each other. Beyond the small buffer results discussed in this paper, this framework introduces a large range of important and unexplored scalings which will require different techniques and are beyond the scope of this paper. 

The remainder of the paper is organised as follows: Section \ref{sec:model} introduces our model and the scaling framework we use throughout the work that follows.  Definitions and assumptions related to the traffic arriving at the system as well as other system parameters are detailed in Section \ref{sec::frame2}. Our main results are stated in Section \ref{sec::results}, and proved in Section \ref{sec::proofs}.

\section{Model and scaling framework}\label{sec:model}
We consider a sequence of $N$ independent single-server queues, such that the $N^{\text{th}}$ system is fed by traffic from $N$ independent identically distributed sources. This relates to the many sources asymptotic \cite{Botvich1995}.

Consider the $N^{\text{th}}$ system in this sequence. We introduce a new scaling parameterization, indexed by $(\alpha, \beta)$. The parameter $\alpha$ is used to control the buffer size scaling, such that in the $N^{\text{th}}$ system the buffer is size $N^{\alpha}B$. For  $\alpha >1$ the buffer grows faster than the number of sources, whereas for $\alpha <1$ the buffer grows slower than the number of sources. $\beta$ is used to control the excess service rate above the total arrival rate,  such that in the $N^{\text{th}}$ system the excess service capacity is  $N^\beta C$. For $\beta<1$ the load increases to one as the number of sources increases; the heavily loaded case. For $\beta > 1$ the load tends to zero as $N$ increases giving the lightly loaded scenario. 

The relationship between $\alpha$ and $\beta$ and the division of the parameter space is shown in Figure \ref{fig: division of parameter space}.
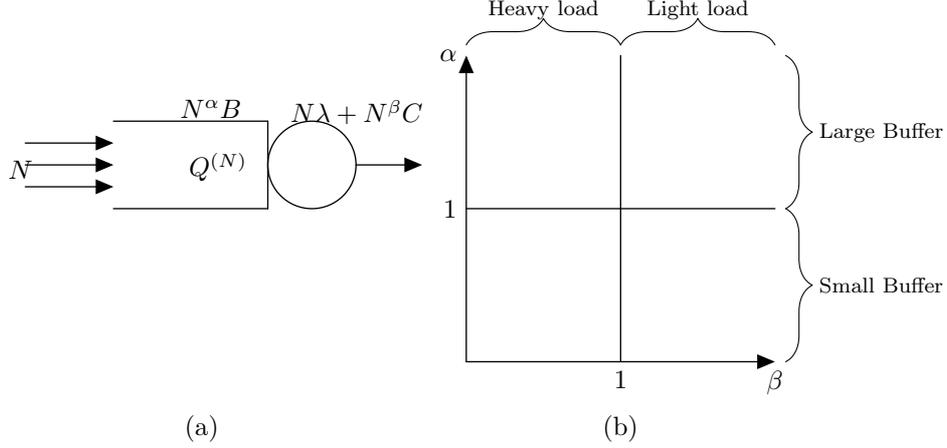
\begin{figure}[htbp]
	\centering
	\begin{tikzpicture}[line cap=round,line join=round,>=triangle 45,x=1cm,y=1cm,scale=0.58]
	\draw [line width=0.5pt] (0,7.5)-- (3.5,7.5);
	\draw [line width=0.5pt] (3.5,7.5)-- (3.5,5.5);
	\draw [line width=0.5pt] (3.5,5.5)-- (0,5.5);
	\draw [line width=0.5pt] (4.5,6.5) circle (1cm);
	\draw [->,line width=0.5pt] (5.5,6.5) -- (7,6.5);
	\draw [->,line width=0.5pt] (-2,6) -- (0,6);
	\draw [->,line width=0.5pt] (-2,6.5) -- (0,6.5);
	\draw [->,line width=0.5pt] (-2,7) -- (0,7);
	\draw [->,line width=0.5pt] (8,2) -- (8,9);
	\draw [->,line width=0.5pt] (8,2) -- (15,2);
	\draw [line width=0.5pt] (8,5.5)-- (15,5.5);
	\draw [line width=0.5pt] (11.5,2)-- (11.5,9);
	\draw [decorate,decoration={brace,amplitude=10pt,mirror,raise=4pt},yshift=0pt]
	(15,2) -- (15,5.5) node [black,midway,xshift=1.4cm] {\footnotesize
		Small Buffer};
	\draw [decorate,decoration={brace,amplitude=10pt,mirror,raise=4pt},yshift=0pt]
	(15,5.5) -- (15,9) node [black,midway,xshift=1.4cm] {\footnotesize
		Large Buffer};
	\draw [decorate,decoration={brace,amplitude=10pt,raise=4pt},yshift=0pt]
	(8,9) -- (11.5,9) node [black,midway,xshift=0cm, yshift=0.6cm] {\footnotesize
		Heavy load};
	\draw [decorate,decoration={brace,amplitude=10pt,raise=4pt},yshift=0pt]
	(11.5,9) -- (15,9) node [black,midway,xshift=0cm, yshift=0.6cm] {\footnotesize
		Light load};
	\draw (8,9) node[anchor=east] {$\alpha$};
	\draw (8,5.5) node[anchor=east] {$1$};
	\draw (11.5,2) node[anchor=north] {$1$};
	\draw (11.5,1) node[anchor=north] {(b)};
	\draw (2,1) node[anchor=north] {(a)};
	\draw (15,2) node[anchor=north] {$\beta$};
	\draw (1.5,7) node[anchor=north west] {$Q^{(N)}$};
	\draw (1.3,8.2) node[anchor=north west] {$N^{\alpha}B$};
	\draw (3.8,8.2) node[anchor=north west] {$N\lambda + N^\beta C$};
	\draw (-2.6,6.8) node[anchor=north west] {$N$};
	\end{tikzpicture}
	\caption{Summary of parameterization of the many flows asymptotic. (a) shows the parameters in the $N^{\text{th}}$ system. (b) is the parameter space for $(\alpha, \beta)$ and the representation of the regions relating to the various scenarios.}
	\label{fig: division of parameter space}
\end{figure}

 We denote the traffic from source $i \, (1 \leq i \leq N)$ as  $A^{(i)}$, where $A^{(i)}_t$ represents the total amount of traffic emitted by source $i$ in the interval $[0, t]$. We assume that each $A^{(i)}_{t}$ has the same distribution as a simple stationary point process $A$ that satisfies $\mathbb{E}[A(0,t)] = \lambda t$ for each $t>0$ for some $\lambda > 0$. The stability condition $C>0$ ensures that the queues do not grow unboundedly. The superposed process $A^{\oplus N} = \sum_{i=1}^{N} A^{(i)}$ represents the aggregate arrival from all $N$ sources. Arriving jobs that cannot be processed immediately are stored in a buffer, assumed to be infinite. We also assume that each job has the same size, and, without loss of generality, set its processing requirement to be one unit, assuming a service rate of $N\lambda +N^{\beta}C$. We use the steady state probability of the unfinished work exceeding a certain level $N^{\alpha}B$ as a surrogate for the steady state buffer overflow probability in a buffer of size $N^{\alpha} B$. Let $\Qn$ denote the stationary unfinished work in this system, when the number of sources is $N$, the arrival process of each source is distributed according to $A$, and the processing rate is $N\lambda + N^{\beta}C$. Under these conditions, $\Qn$ is given by
\begin{equation}
    \Qn = \sup_{t \in [0, \infty)} \left[A^{\oplus N} (0,t) - N \lambda t - N^\beta Ct \right]\,.
\label{eqn::scaled queue}
\end{equation}

To investigate the sample path scalings, we need to find a scaled process of the arrivals, $\tilde{A}^N_{\alpha, \beta}$, such that
\begin{align*}
\P(f_C(\tilde{A}^N_{\alpha, \beta}) > B) = \P(\Qn > N^\alpha B)\,,
\end{align*}
where $f_C$ is the queuing map, defined by
\begin{align*}
f_C(x) = \sup_{t>0} (x(t) - Ct)\,.
\end{align*}
Using the continuous form of Loynes' scheme for stationary queue length \cite{Loynes1962} we obtain
\begin{align*}
\P(\Qn > N^\alpha B) = \P \left(\sup_{t \geq 0} \sum_{i=1}^{N} A_i(0,t) - (N \lambda + N^\beta C) t > N^\alpha B \right)\,.
\end{align*}
We can re-arrange this to obtain
\begin{align*}
\P(\Qn > N^\alpha B) & = \P \left(\sup_{t \geq 0} \frac{\sum_{i=1}^N A_i(0,t)}{N^\alpha} - N^{1-\alpha} \lambda t - N^{\beta - \alpha} Ct > B \right) \\
								& = \P \left(\sup_{t' \geq 0} \frac{\sum_{i=1}^N A_i(0,N^{\alpha - \beta}t')}{N^\alpha} - N^{1-\beta} \lambda t' - Ct' > B \right) \\
								& = \P \left(f_C \left( \left \lbrace \frac{\sum_{i=1}^N A_i(0,N^{\alpha - \beta}t')}{N^\alpha} -  N^{1-\beta} \lambda t \right \rbrace_{t \geq 0}\right) > B \right),
\end{align*}
where $t' = tN^{\beta - \alpha}$. The natural scaled process to consider is therefore
\begin{equation}
\left \lbrace \tilde{A}^N_{\alpha, \beta} \right \rbrace_{t>0} = \left \lbrace \frac{\sum_{i=1}^N A_i(0,N^{\alpha - \beta}t')}{N^\alpha} -  N^{1-\beta} \lambda t' \right \rbrace_{t>0} \,.
\label{eqn::scaling}
\end{equation}

We will examine sample path large deviations principles in the following five cases: 
\begin{enumerate}[label=(\roman*)]
\item $\alpha=\beta=1$: original large deviations asymptotic,
\item $0<\alpha<\beta = 1$: small buffer large deviations asymptotic,
\item $1/2<\alpha=\beta <1$: original moderate deviations  asymptotic,
\item $\alpha < \beta<1$, $\alpha + \beta >1$: small buffer moderate deviations asymptotic,
\item $0<\alpha<1$, $\beta>1$: large deviations for light-load.
\end{enumerate}

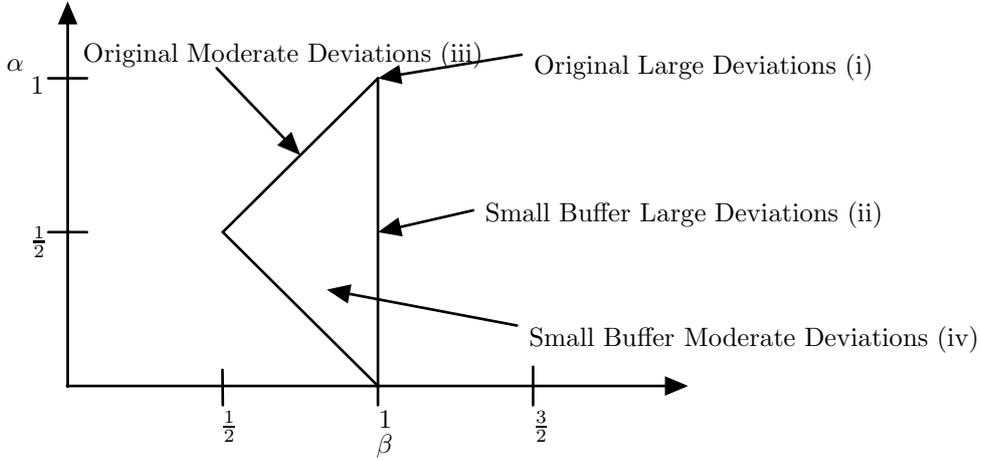
\begin{figure}[htb]
\begin{center}
\begin{tikzpicture}[line cap=round,line join=round,>=triangle 45,x=1cm,y=1cm, scale=0.51]
\draw [->,line width=1pt] (0,0) -- (0,10);
\draw [->,line width=1pt] (0,0) -- (16,0);
\draw [line width=1pt] (8,0)-- (4,4);
\draw [line width=1pt] (4,4)-- (8,8);
\draw [line width=1pt] (8,8)-- (8,0);
\draw [->,line width=1pt] (10.456575291389697,4.567072079319329) -- (8,4);
\draw [->,line width=1pt] (11.584907915144662,1.5711544231423582) -- (6.682497205036888,2.5049469393533625);
\draw [->,line width=1pt] (3.842211634895082,8.263334122654554) -- (6,6);
\draw [->,line width=1pt] (11.701631979671037,8.61350631623368) -- (8,8);
\draw (11.776123988802633,8.884464840823513) node[anchor=north west] {Original Large Deviations (i)};
\draw (10.521380906617248,5.030611088396964) node[anchor=north west] {Small Buffer Large Deviations (ii)};
\draw (11.65662464764212,1.8041288770631085) node[anchor=north west] {Small Buffer Moderate Deviations (iv)};
\draw (0.1548130609427649,9.183213193724796) node[anchor=north west] {Original Moderate Deviations (iii)};
\draw [line width=1pt] (3.9909545596370024,0.42011463460010173)-- (4,-0.5);
\draw [line width=1pt] (8,0)-- (8,-0.5);
\draw [line width=1pt] (12,0.5)-- (12,-0.5);
\draw [line width=1pt] (-0.5,4)-- (0.5,4);
\draw [line width=1pt] (-0.5,8)-- (0.5,8);
\draw (-1.2,4.4) node[anchor=north west] {$\frac{1}{2}$};
\draw (-1.2,8.3) node[anchor=north west] {$1$};
\draw (3.7,-0.4) node[anchor=north west] {$\frac{1}{2}$};
\draw (7.772896059925454,-0.4) node[anchor=north west] {$1$};
\draw (11.746249153512505,-0.4) node[anchor=north west] {$\frac{3}{2}$};
\draw (7.713146389345198,-1.0041056402089505) node[anchor=north west] {$\beta$};
\draw (-1.816926068205696,8.735090664372871) node[anchor=north west] {$\alpha$};
\begin{scriptsize}
\draw [fill=black] (0,0) circle (1pt);
\draw [fill=black] (8,0) circle (1pt);
\end{scriptsize}
\end{tikzpicture}
\end{center}
\caption{Regions in the parameter space covered by the scalings (i)--(iv).}
\label{fig::ppcases}
\end{figure}

Figure \ref{fig::ppcases} illustrates the relationship between the scalings (i)--(iv). Both the scalings in (i) and (iii) have been considered previously for many traffic processes for queue length \cite{Weiss1986, Botvich1995, Courcoubetis1996, Mandjes2001, Simonian1994} and sample path results in discrete time \cite{Wischik2001, moddev}. Here we extend these results to sample paths in continuous time.

It is natural to consider scalings (i) and (iii) together, and (ii), (iv) and (v) together because of the timescale upon which the events of interest occur. In the first two scalings the timescales stay constant as we increase the number of flows. In contrast, in the latter three cases they converge to 0, so overflow events happen quickly, which leads to a simple rate function and insensitivity of results to covariance structure. Separately in scalings (iii) and (iv) we see the heavy loads leading to Gaussian structure, and so the proofs reflect this.

\section{Traffic assumptions and associated spaces}\label{sec::frame2}
In this section we introduce the various processes of interest in the queuing systems we study. We begin with the arrivals processes, which we assume to be marked point processes.

In many practical applications, the use of a marked point process as a traffic model is sensible since traffic often arrives as units rather than a continuous stream, for example, packets in the Internet or customers in a shop. The marking allows the study of general systems where arrivals can bring more than one customer. An example is computer protocols where multiple packets can be transmitted together. Alternatively,  the marks can represent service times of customers and queue length can represent residual work load using a deterministic server.

Let $\mathbb{X}$ be a collection of independent and identically distributed simple stationary point processes \cite[Definition 3.3.II]{Daley2008}, with mean $\lambda$. Let $\mathbb{X}^{(i)}$ be the $i^\text{th}$ point process and $\mathbb{X}^{(i)}(t_1, t_2)$ be the number of points in the set $(-t_2, -t_1]$ for the $i^{\text{th}}$ process, so that $\mathbb{E}(X^{(i)}(0,t)) = \lambda t$ by stationarity. We associate a collection of stationary, positive, identically distributed random variables $\mathbb{Y}^{(i)} \in \mathbb{R}^{\infty}$ with each process, and let $Y^{(i)}_j$ be the $j^{\text{th}}$ random variable of the $i^{\text{th}}$ collection. Each collection of random variables is independent of the other collections and of all the point processes, and all the collections are identically distributed. Let ${Y}$ represent a generic random variable ${Y}_j^{(i)}$. We define the marked point process $A^{(i)}$, representing arrivals from source $i$, where we associate $Y_j^{(i)}$ with the $j^{\text{th}}$ point, by
\begin{align*}
    A^{(i)}(0,t) = \sum_{j=1}^{X^{(i)}(0,t)} Y_j^{(i)}\,,
\end{align*}
and 
\begin{align*}
    A^{(i)}(t_1, t_2) = A^{(i)}(0,t_2) - A^{(i)}(0,t_1)\,.
\end{align*}
We define the aggregates $A^{\oplus N}$ and $X^{\oplus N}$ by $A^{\oplus N}(t_1, t_2) {=} \sum_{i=0}^N A^{(i)}(t_1, t_2)$ and $X^{\oplus N}(t_1, t_2) {=} \sum_{i=0}^N X^{(i)}(t_1, t_2)$, respectively, the aggregate of $N$ marked point processes.

The server is a deterministic server such that the $N^\text{th}$ system, which has $N$ input sources, has service rate $N\lambda + N^\beta C$ and buffer size $N^\alpha B$. We let $Q^{(N)}$ be the stationary queue length in the $N^\text{th}$ system. This is given by 
\begin{align*}
    \Qn = \sup_{t \in [0, \infty)} \left[A^{\oplus N} (0,t) - N \lambda t - N^\beta Ct \right]\,.
\end{align*}

In addition, we will need the following definitions:
For
$x\in \mathbb{R}$ and $t\in [0,\infty)$, 
\begin{displaymath}
\Psi (x,t) = \sup_{\theta \in \mathbb{R}} \left[\theta x - t^{-1} \log \mathbb{E}[ e^{\theta A(0,t)}]\right]\,, \quad \Psi_\infty^1(x) = \liminf_{t \rightarrow \infty} \frac{t \Psi(x,t)}{\log t}\,.
\end{displaymath}
When considering the moderate deviations results, we will also need tighter bounds on the behaviour of the arrival process at large timescales. To achieve this we will need to consider the following limits:
 \begin{displaymath}
 \Psi_{\infty,d\rightarrow t}^2  = \liminf_{t \rightarrow \infty} \liminf_{d\rightarrow 0} \frac{t \Psi(\lambda\mathbb{E}(Y) + d ,t)}{d^2\log t}
\,,\end{displaymath}
 and 
  \begin{displaymath}
   \Psi_{\infty,t\rightarrow d}^2  = \liminf_{d\rightarrow 0} \liminf_{t \rightarrow \infty} \frac{t \Psi(\lambda\mathbb{E}(Y) + d ,t)}{d^2\log t}\,.
 \end{displaymath}

We will also need the moment generating function of $Y$, which we denote $M(\theta) {=} \mathbb{E}(e^{\theta Y})$, and the log moment generating function of $A(0,t)$, which we denote $\Lambda_t(\theta) {=} \log \mathbb{E}(e^{\theta A(0,t)})$. In addition, we define the set on which the moment generating function is finite, $\mathcal{D}_M {=} \lbrace \theta \in \mathbb{R}: M(\theta) < \infty \rbrace$. Finally, let $\mathcal{D}_y$ be the space of cadlag functions, $x: \mathbb{R}^+ \mapsto \mathbb{R}^+$ for which $x(0) = 0$ and 
\begin{align*}
    \lim_{t \to \infty} \frac{x(t)}{1+t} = y\,.
\end{align*}

While considering sample paths we will need to characterize compact sets on which the measure associated with the process is large. Marked point process naturally live in the space of cadlag functions in which it is difficult to classify compact sets. So we instead often make use of a linear interpolation of the marked point process which lives in the space of continuous functions. We define this as follows:

\begin{defn} \label{def:linerinterp}
Let $\tau_N^-(t)$ be the time of the last point before or at
$t$ of $X^{\oplus N}$, and  $\tau_N^+(t)$ be the time of the next point after $t$ of $X^{\oplus N}$, and let $\zeta_N(t)$ be its associated mark.
Then let $\bar{A}^{\oplus N}$ be the
polygonal approximation of $A^{\oplus N}$ defined by
\begin{displaymath}
\bar{A}^{\oplus N}(0,t)= A^{\oplus N}(0,t)
+\frac{t-\tau_N^-(t)}{\tau_N^+(t)-\tau_N^-(t)}\zeta_N(t).
\end{displaymath}
We define the scaled version, $\tilde{\bar{A}}^N_{\alpha,\beta}$, in the obvious way.
\end{defn}

Our results will be proved in the space of cadlag functions with a given long run mean rate $y$, $\mathcal{D}_y$, with the topology induced by the scaled uniform norm $\|\cdot\|_s$; see \cite{MR0232428}, \cite[page 29]{Ganesh2004} . In addition, in stating the results we utilize the subspace of absolutely continuous functions and the reproducing kernel Hilbert space $R_v$ for a given variance function $v$ \cite{Berlinet2004}.  We let $\mathscr{AC}$ denote the set of absolutely continuous functions on an appropriate space as required.

\section{Results} \label{sec::results}

We use this section to state out main results and the assumptions under which they hold; proofs are deferred until Section \ref{sec::proofs}. We begin with case (i), $\alpha=\beta =1$; for work in this regime in the discrete time setting, see Wischik \cite{Wischik2001}.  Our results in this continuous time setting will be established under the following assumptions:
\begin{assumption} \label{ass::ldlb}
$\,$
\begin{enumerate}
\item  There exist $\theta_0 >0$ and $K<\infty$ such that
\begin{displaymath}
\lim_{t\rightarrow 0}t^{-1}
\mathbb{E}\left[e^{\theta_0 A(0,t)}1_{\{X(0,t)>K\}}\right]=0.
\end{displaymath}
\item  $\Psi^1_\infty(x) > 0$ for all $x \neq \lambda \mathbb{E}(Y)$.
\item  $\Lambda_{t} (\theta) < \infty$ for all $t >0$ and $\theta \in \mathbb{R}$. 
\item $\{\bar{A}^{\oplus N}( \cdot)/N \}_{t\in [0,T)}$ is exponentially tight \cite[Theorem 4.2.10]{Dembo1998} in the space of continuous functions on the interval $[0,t)$ with the scaled uniform norm.
\end{enumerate}
\end{assumption}
With these assumptions we obtain the following:
\begin{thm}\label{thm::ldlb}
Let $A$ be a stationary marked point process.
Under Assumptions \ref{ass::ldlb}, the sequence  $\tilde{A}^N_{1,1}$ satisfies a sample path large deviations principle in the space $\mathcal{D}_{0}$ with the scaled uniform norm $||x||_s$. This has rate $N$, and good rate function $I_{1,1}(x)$, where
\begin{equation}\label{eqn::ldlbrf}
I_{1,1}(x)= \sup_{j \in \mathscr{J}, T\in\mathbb{R}^+} \Lambda^*_{j,T}(x)\,,
\end{equation}
where $\mathscr{J}$ is the collection of all ordered finite subsets
of $(0,1]$ with $j_0=0$, and
$$ \Lambda^*_{j,t}(x) = \sup_{\theta \in \mathbb{R}^{|j|}} \left \{ \sum_{i=1}^{|j|} (x(t j_i) - x(t j_{i-1}) ) \theta_i -\Lambda_{j,t}(x) \right \}\,,$$
with
$$\Lambda_{j,t}(x) = \log \mathbb{E} \left( e^{\sum_{i=1}^{\lvert j \rvert}  \theta_i A(j_{i-1} t,  j_{i} t )} - \lambda\mathbb{E}(Y)\sum_{i=1}^{\lvert j \rvert}\theta_i t(j_i -j_{i-1})\right).$$
\end{thm}

We now turn to results for small buffers which involve short timescales. 
We start by considering scaling (ii), where $0<\alpha<\beta=1$. Results here have previously been established by \cite{Cruise2009a}; we state them here (together with an outline of the proof in Section \ref{sec::proofs}) to illustrate how this fits into the framework which we introduce.  In this setting we need the following assumptions:
\begin{assumption}\label{ass::ldsb}
$\,$
\begin{enumerate}
\item  There exist $\theta_0 >0$ and $K<\infty$ such that
\begin{equation*}
\lim_{t\rightarrow 0}t^{-1}
\mathbb{E}\left[e^{\theta_0 A(0,t)}1_{\{X(0,t)>K\}}\right]=0\,.
\end{equation*}
 \item  $\Psi^1_\infty(x) > 0$ for all $x \neq \lambda \mathbb{E}(Y)$.
\item  $M (\theta)< \infty$ for all $\theta \in \mathbb{R}$.
\end{enumerate}
\end{assumption}
With these we have:
\begin{thm}\label{thm::ldsb}
Let $A$ be a stationary marked point process.
Under Assumptions \ref{ass::ldsb} and given $0<\alpha<1$, the sequence  $\tilde{A}^N_{\alpha,1}$ satisfies a sample path large deviations principle in the space $\mathcal{D}_{0}$ with the scaled uniform norm $||x||_s$. This has rate $N^{\alpha}$, and good rate function $I_{\alpha,1}(x)$, where
\begin{equation}\label{eqn::ldsbrf}
I_{\alpha,1}(x)= \begin{cases}
\int_0^{\infty} \Omega^*(\dot{x}+\lambda \mathbb{E}(Y)) dt &\textrm{ if $x(0)=0$ and $x\in\mathscr{AC}$\,,} \\
\infty &\textrm{otherwise\,,}
\end{cases}
\end{equation}
and
\begin{equation}\label{gamma}
\Omega^*(y)=\sup_{\theta \in \mathbb{R}} \big[ \theta y - \lambda(M(\theta)-1) \big]\,.
\end{equation}
\end{thm}

Next we look at the moderate deviations case (iii), as first examined in discrete time by Wischik \cite{moddev}; here $1/2 <\alpha=\beta<1$. We need to modify the assumptions under which we work slightly in our continuous time setting.
\begin{assumption} \label{ass::mdlb}
$\,$
\begin{enumerate}
\item  There exist $\theta_0 >0$ and $K<\infty$ such that
\begin{displaymath}
\lim_{t\rightarrow 0}t^{-1}
\mathbb{E}\left[e^{\theta_0 A(0,t)}1_{\{X(0,t)>K\}}\right]=0\,.
\end{displaymath}
\item $\Psi_{\infty,d\rightarrow t}^2 > 0$  and $\Psi_{\infty,t\rightarrow d}^2 > 0$.
\item There exists $\theta^*>0$  such that, for all $\theta \in [0,\theta^*)$ and $t\in\mathbb{R}^+$, $\Lambda_{t} (\theta) < \infty$.
\item For $\alpha \in (1/2,1)$, $\{\tilde{\bar{A}}_{\alpha,\alpha}(\cdot) \}_{t\in [0,T)}$ is exponentially tight in the space of continuous  functions on the interval $[0,t)$ with the scaled uniform norm.
\end{enumerate}
\end{assumption}
We then obtain the following:
\begin{thm}\label{thm::mdlb}
Let $A$ be a stationary marked point process with continuous variance function $v$.
Under Assumptions \ref{ass::mdlb}, for $1/2<\alpha<1$ the sequence  $ \tilde{A}^N_{\alpha,\alpha}$ satisfies a sample path large deviations principle in the space $\mathcal{D}_{0}$ with the scaled uniform norm $||x||_s$. This has rate $N^{2 \alpha -1}$, and good rate function $I_{\alpha,\alpha}(x)$, where
\begin{equation}\label{eqn::mdlbrf}
I_{\alpha,\alpha}(x)= \begin{cases} \frac{1}{2} \| x\|^2_{R_v}  & \text{if $x\in R_{v}$\,,} \\
\infty & \text{otherwise\,,}
\end{cases} 
\end{equation}
where $R_v$ is the reproducing kernel Hilbert space associated  with variance function $v$. 
\end{thm}

Next we present moderate deviations (iv) for these fast timescales where $\alpha<\beta<1$ and $\alpha+\beta>1$. These scalings lie between the small buffer result and the first moderate deviations result.  In this setting we need the following assumptions:
\begin{assumption}\label{ass::mdsb}
$\,$
\begin{enumerate}
\item  There exist $\theta_0 >0$ and $K<\infty$ such that
\begin{displaymath}
\lim_{t\rightarrow 0}t^{-1}
\mathbb{E}\left[e^{\theta_0 A(0,t)}1_{\{X(0,t)>K\}}\right]=0\,.
\end{displaymath} 
 \item  $\Psi_{\infty,d\rightarrow t}^2(x) > 0$  and $\Psi_{\infty,t\rightarrow d}^2(x) > 0$.
\item $M (\theta)< \infty$ for all $\theta$ in a neighbourhood of 0.
\end{enumerate}
\end{assumption}
The result we obtain here is the following:
\begin{thm}\label{thm::mdsb}
Let $A$ be a stationary marked point process.
Under Assumptions \ref{ass::mdsb} and given $\alpha < \beta<1$  such that $\alpha +\beta>1$, the sequence  $\tilde{A}^N_{\alpha,\beta}$ satisfies a sample path large deviations principle in the space $\mathcal{D}_{0}$ with the scaled uniform norm $||x||_s$. This has rate $N^{\alpha+\beta -1}$, and good rate function $I_{\alpha, \beta}(x)$, where
\begin{equation}\label{eqn::mdsbrf}
I_{\alpha, \beta}(x)= \begin{cases}
\int_0^{\infty} \frac{\dot{x}^2}{2 \lambda \mathbb{E}(Y^2)} dt &\textrm{ if $x(0)=0$ and $x\in\mathscr{AC}$\,,} \\
\infty &\textrm{otherwise\,.}
\end{cases}
\end{equation}
\end{thm}

%
%
The final case we consider is (v), large deviations for fast timescales where $0<\alpha<1$ and $\beta > 1$. In this setting we need the following assumptions:
\begin{assumption}\label{ass::ldbg1}
$\,$
\begin{enumerate}
\item  There exist $\theta_0 >0$ and $K<\infty$ such that
\begin{displaymath}
\lim_{t\rightarrow 0}t^{-1}
\mathbb{E}\left[e^{\theta_0 A(0,t)}1_{\{X(0,t)>K\}}\right]=0\,.
\end{displaymath} 
\item  $\Psi^1_\infty(x) > 0$ for all $x \neq \lambda \mathbb{E}(Y)$.
\item  $M (\theta)< \infty$ for all $\theta \in \mathbb{R}$.
\end{enumerate}
\end{assumption}
Our result here is the following:
\begin{thm}\label{thm::ldbg1}
Let $A$ be a stationary marked point process.
Under Assumption \ref{ass::ldbg1} and given $0<\alpha <1$  and $\beta>1$, the sequence  $\tilde{A}^N_{\alpha,\beta}$ satisfies a sample path large deviations principle in the space $\mathcal{D}_{0}$ with the scaled uniform norm $||x||_s$. This has rate $N^{\alpha} \log N$, and good rate function $I_{\alpha, \beta}(x)$, where
\begin{equation}
    I_{\alpha, \beta}(x) = \begin{cases}
\int_0^{\beta - 1} \dot{x} dt &\textrm{ if $x(0)=0$ and $x\in\mathscr{AC}$\,,} \\
\infty &\textrm{otherwise\,.}
\end{cases}
\end{equation}
\end{thm}

%
To understand the assumptions we impose on the traffic process we examine how they lead to properties of the processes required in the proof. The systems studied in this paper are already stable by virtue of the nature of the scaling. We also require the amount of traffic produced in a small time period to be small, and this is guaranteed by the first of each of the assumptions. In addition we require the probability of a source sending at a rate greater than the service rate of the queue for long periods to be small. Often when considering large deviations limits this is dealt with by examining whether $\lim_{t \to \infty}\Psi(C,t)$ is non-zero. However we consider a scaled form $\Psi_{\infty}^{1}(x)$ which is a weaker condition.

It is worth noting that for three of these results (Theorems \ref{thm::mdlb}, \ref{thm::ldsb} and \ref{thm::mdsb}) the associated rate functions are the same as for previously studied stochastic processes. Firstly, in Theorem \ref{thm::mdlb} the rate function in (\ref{eqn::mdlbrf}) is that of a stationary Gaussian process with variance function $v$, and of the generalized Schilder's theorem \cite[Theorem 5.2.3]{Dembo1998}. In comparison to this, we have that in the small buffer large deviations case (Theorem \ref{thm::ldsb}) the rate function (\ref{eqn::ldsbrf}) is that of a marked Poisson process with mean rate $\lambda$ and independent marks with distribution $Y$.  Finally, for the moderate deviations small buffer case (Theorem \ref{thm::mdsb}) we find that the rate function (\ref{eqn::mdsbrf}) is that of a Brownian motion with variance parameter $\lambda \mathbb{E}(Y^2)$. In these latter two cases the rate function only depends on the mean rate of the point process and either the moment generating function of the marks or the second moment of the marks. This allows easy calculations in many circumstances, and these calculations are robust to changes in the underlying point process. This is useful from a modelling perspective, as we only need to estimate the mean rate of the point process and properties of the marks to provide useful estimates of tail probabilities. 

The move from large deviations to moderate deviations requires a strengthening of the assumption on the long run behaviour, but we are able to weaken the assumption on the moment generating functions, as in the limit the second moment dominates. In addition, for the large buffer cases the conditions on the log moment generating functions of the processes (our Assumptions \ref{ass::ldlb}.3 and \ref{ass::mdlb}.3) imply the associated condition on the moment generating function of the marks $M(\theta)$. Finally, in the large buffer cases we have an extra condition which guarantees the processes are exponentially tight on finite timescales but we are able to drop this in the small buffer setting. This is because the limiting process is relatively insensitive to the original process in the small buffer case.

\section{Proofs of theorems} \label{sec::proofs}

In this section we prove the results we stated in Section \ref{sec::results}. These proofs are split into three sections: some preliminary lemmas, a framework for the large deviations results, and a final step combining these. Section \ref{sec::lemmas} contains a series of lemmas on the marked point processes in the various scalings we consider. This includes two limits for the log moment generating function, and also bounds for behaviour at long timescales in the various scenarios. In Section \ref{sec::framework} we set up a general framework of lemmas for proving our large deviations results. We begin by proving a large deviations principle over a fixed time interval and then extending this to infinite time. Finally, for each of the scalings of interest we show how to apply the previous lemmas to obtain the desired results.

\subsection{Preliminary lemmas}\label{sec::lemmas}

\subsubsection{Lemmas for large deviation results (\texorpdfstring{$\beta \geq 1$}{})}

\begin{lemma} \label{lem::pple}
Suppose $A$ is  a marked point process  which obeys Assumptions \ref{ass::ldlb} or \ref{ass::ldsb}. Given $0=j_0<j_1<j_2< \cdots <j_{n-1}<j_n = 1$ and $t \in \mathbb{R}$, define the vector $\mathbf{A}^{\mathbf{j},t}$ by $\mathbf{A}^{\mathbf{j},t}_i = A(t j_{i-1}, t j_i)$. Then there exists $\mathbf{\theta}^* >0$ such that, uniformly for all $\theta\in\{ \theta \in \mathbb{R}^n: \theta_i \in [0,\theta^*]\}$,
\begin{equation}\label{eqn::explimit2}
\lim_{t \rightarrow 0} \frac{1}{t} \log \mathbb{E}\left[ e^{\langle\theta, \mathbf{A}^{\mathbf{j},t}\rangle} \right]=\lambda \sum_{i=1}^n (j_{i}-j_{i-1})(M(\theta_i) -1)\,.
\end{equation}
\end{lemma}
\begin{proof}
Firstly, we have $\mathbb{P}(X(0,t)=1) =\lambda t +o(t)$ and $\mathbb{P}(X(0,t) \geq 2) =o(t)$. These follow directly from
\cite[Propositions 3.3.I, 3.3.IV and 3.3.V]{Daley2008}. Now let
$p_t(k) = \mathbb{P}(X(0,t)=k)$ for $k\in \mathbb{Z^+}$. As we choose
$\theta^*>0$ and $K<\infty$ to be such that the  Assumptions \ref{ass::ldlb} or \ref{ass::ldsb} are
satisfied with $\theta^* \leq \theta_0$ and $M(\theta^*)<\infty$,
\begin{multline*}
\mathbb{E}\left[ e^{\langle\theta, \mathbf{A}^{\mathbf{j},t}\rangle} \right] -1 =p_t(0) +p_t(1)\mathbb{E}\left[ e^{\langle\theta, \mathbf{A}^{\mathbf{j},t}\rangle} | X(0,t)=1 \right]\\ 
+ \sum_{k=2}^K p_t(k) \mathbb{E}\left[ e^{\langle\theta, \mathbf{A}^{\mathbf{j},t}\rangle} | X(0,t)=k \right] +\mathbb{E}\left[ e^{\langle\theta, \mathbf{A}^{\mathbf{j},t}\rangle} 1_{X(0,t) >K} \right] -1\,.
\end{multline*}
Since $\mathbb{E}\left[ e^{\langle\theta, \mathbf{A}^{\mathbf{j},t}\rangle} |
X(0,t)=k \right] $ has no dependence on $t$ (as the $Y_i$ are independent of the point process), we use the previous probability
approximations to get
\[
\mathbb{E}\left[ e^{\langle\theta, \mathbf{A}^{\mathbf{j},t}\rangle} \right] -1 
=-\lambda t + \lambda t\mathbb{E}\left[ e^{\langle\theta, \mathbf{A}^{\mathbf{j},t}\rangle} | X(0,t)=1 \right] +\mathbb{E}\left[ e^{\langle\theta, \mathbf{A}^{\mathbf{j},t}\rangle} 1_{X(0,t) >K} \right]+o(t)\,.
\]
Consider $\mathbb{E}\left[ e^{\langle\theta, \mathbf{A}^{\mathbf{j},t}\rangle} | X(0,t)=1 \right]$. By conditioning the sub-interval containing the point, we get
\begin{displaymath}
\mathbb{E}\left[ e^{\langle\theta, \mathbf{A}^{\mathbf{j},t}\rangle} | X(0,t)=1 \right] = \sum_{i=1}^n \mathbb{P}(X( t j_{i-1},t j_i)=1 | X(0,t)=1) M(\theta_i)\,.
\end{displaymath}
Also,
\begin{displaymath}
0\leq \mathbb{E}\left[ e^{\langle\theta, \mathbf{A}^{\mathbf{j},t}\rangle} 1_{X(0,t) >K} \right] \leq \mathbb{E}\left[ e^{\theta^* A(0,t)} 1_{X(0,t) >K} \right] \leq \mathbb{E}\left[ e^{\theta_0 A(0,t)} 1_{X(0,t) >K} \right] \,.
\end{displaymath}
Using these and our assumptions, we get
\[
\lim_{t\rightarrow 0} \frac{1}{t} \left( \mathbb{E}\left[ e^{\langle\theta, \mathbf{A}^{\mathbf{j},t}\rangle} \right] -1 \right)
= -\lambda+\lambda \lim_{t\rightarrow 0} \left(\sum_{i=1}^n \mathbb{P}(X( t j_{i-1},t j_i)=1 | X(0,t)=1) M(\theta_i) \right)\,,
\]
and
\begin{displaymath}
\lim_{t\rightarrow 0} \frac{1}{t} \left( \mathbb{E}\left[ e^{\langle\theta, \mathbf{A}^{\mathbf{j},t}\rangle} \right] -1 \right)^2 =0\,.
\end{displaymath}
We now need to find $\lim_{t\rightarrow 0} \mathbb{P}(X( t j_{i-1},t j_i)=1 | X(0,t)=1)$.  Firstly,
\begin{displaymath}
\mathbb{P}(X( 0,t/2)=1 | X(0,t)=1) = \frac{\mathbb{P}(X( 0,t/2)=1  , X( t/2,t)=0)}{\mathbb{P}(X( 0,t)=1)}\,.
\end{displaymath}
Since \mbox{$\mathbb{P}(X(0,t/2)=1)= \lambda t/2 +o(t)$,} and $\mathbb{P}(X(t/2,t)=1)= \lambda t /2+o(t/2)$, we have
\begin{displaymath}
\lim_{t\rightarrow 0} \mathbb{P}(X( 0,t/2)=1 | X(0,t)=1) = \lim_{t\rightarrow 0} \frac{\lambda\frac{t}{2} +o(t)}{\lambda t +o(t)} =\frac{1}{2}\,.
\end{displaymath}
A similar argument gives
\begin{equation}\label{eqn::shortprob}
\lim_{t\rightarrow 0}\mathbb{P}(X( t j_{i-1},t j_i)=1 | X(0,t)=1) = j_i -j_{i-1}\,.
\end{equation}
Thus,
\begin{equation}\label{eq::explimit}
\lim_{t\rightarrow 0} \frac{1}{t} \left( \mathbb{E}\left[ e^{\langle\theta, \mathbf{A}^{\mathbf{j},t}\rangle} \right] -1 \right) = -\lambda+\lambda  \left(\sum_{i=1}^n (j_i- j_{i-1}) M(\theta_i) \right)\,.
\end{equation}
Finally, we have that $x-x^2/2 \leq \log(1+x) \leq x$ for all $x\geq 0$, which gives the desired result.
\end{proof}

\begin{lemma}\label{lem::lt2}
Suppose $A$ is a marked point process such that $\Psi^1_\infty(x)>0$ for $x>\lambda\mathbb{E}(Y)$. Given $B>0$ and $d \in \mathbb{R}^+$, there exists $t_0 <\infty$ such that
\begin{displaymath}
\limsup_{N \rightarrow \infty} \frac{1}{N^{\alpha}} \log
\mathbb{P}\left(\sup_{t>(N^\alpha t_0)/N} [A^{\oplus N}(0,t)-Nxt]\geq N^\alpha B \right)
\leq -d.
\end{displaymath}
\end{lemma}

\begin{proof}
We use the scaling $F^{\oplus N}(0,t) = A^{\oplus N}(0,t/N)$. This gives 
 $$\sup_{t>(N^\alpha t_0)/N} [A^{\oplus N}(0,t)-Nxt] =\sup_{t>N^\alpha  t_0} [F^{\oplus N}(0,t)-xt]\,.$$
Hence
\begin{align}
\nonumber\mathbb{P}\left(\sup_{t>(N^\alpha t_0)/N}[A^{\oplus N}(0,t)-Nxt]\geq N^{\alpha}B\right)
&=\mathbb{P}\left(\sup_{t>N^\alpha t_0} [F^{\oplus N}(0,t)-xt]\geq N^{\alpha}B\right) 
\\ &\leq \mathbb{P}\left(\sup_{t>N^\alpha t_0} [F^{\oplus N}(0,t)-xt]\geq B\right)\,.\label{scaling1}
\end{align}

We choose $\delta \in (0,B)$ such that $s= \delta x^{-1} <1$.
For $l \in \mathbb{Z}^+$, let $t_l = s l$ and $I_l =
[t_l, t_{l+1})$. Since $F^{\oplus N}(0,t)$ is non-decreasing
in $t$ and $B-\delta >0$ we get
\begin{align*}
\mathbb{P} \left(\sup_{t\in [t_l , t_{l+1} ) } [F^{\oplus N}(0,t)-xt]\geq B\right) & \leq \mathbb{P}(F^{\oplus N}(0,t_{l+1})> B +xt_l) \\
&=\mathbb{P}(F^{\oplus N}(0,t_{l+1})> B -\delta +xt_{l+1}) \\ &\leq \mathbb{P}(F^{\oplus N}(0,t_{l+1})> xt_{l+1})\,.
\end{align*}
We can now use the Chernoff bound to see that, for $\theta\geq 0$,
\begin{displaymath}
\mathbb{P}(F^{\oplus N}(0,t_{l})> xt_{l}) \leq e^{-xt_l \theta} \mathbb{E}[e^{\theta F^{\oplus N}(0,t_{l})}]
=e^{-t_l \left( x \theta -N t_l^{-1} \log \mathbb{E}[e^{\theta A(0,t_{l}/N)}]\right) }\,.
\end{displaymath}
Taking the infimum over $\theta \geq 0$ we get $\mathbb{P}(F^{\oplus N}(0,t_{l})> xt_{l}) \leq e^{-t_l \Psi(x,N^{-1} t_l)}$.

Given $t_0\in (s^{-1}, \infty)$, let $L_N \in \mathbb{Z}^+$ be such
that $N^\alpha t_0 \in [t_{L_N} ,t_{L_N +1})$.  Let $N_0 = \inf \{ N: N\tau >L_N \}$.  For $\tau \in (0, \infty)$ and all $N>N_0$, we have that (\ref{scaling1}) is at most
\begin{multline} 
 \mathbb{P} \left(\sup_{t\in [sL_N , s \lfloor N\tau\rfloor ) } [F^{\oplus N}(0,t)-xt]\geq B\right) +\mathbb{P} \left(\sup_{t\in [s \lfloor N\tau\rfloor , \infty ) } [F^{\oplus N}(0,t)-xt]\geq B \right) \\
\leq \sum_{l=L_N+1}^{\lfloor N\tau\rfloor} e^{-t_l \Psi(x,N^{-1} t_l)} + \sum_{l=\lfloor N\tau\rfloor}^{\infty} e^{-t_l \Psi(x,N^{-1} t_l)}\,. 
\label{eq::52}
\end{multline}
We find a bound for $\Psi(x,t)$ using
trivial extensions to Lemma 5 and Corollary 6 of 
\cite{CaoRamanan2002}, which state that there exist $\acute{\tau}>1$, and $ \beta_1,\beta_2>0$ such
that $ \Psi(x,t) \geq \beta_1$ for $t\in[0,\acute{\tau}]$ and $t\Psi(x,t)
/\log t \geq \beta_2$ for $t\in[\acute{\tau},\infty)$. This gives
\begin{equation} \label{firstsum1}
 \sum_{l=L_N+1}^{\lfloor N\acute{\tau}\rfloor} e^{-t_l \Psi(x,N^{-1} t_l)}  \leq \sum_{l=L_N+1}^{\lfloor N\acute{\tau}\rfloor} e^{-t_l \beta_1} \leq \sum_{l=L_N+1}^{\infty} e^{-t_l \beta_1}\,.
\end{equation}
Using the definition of $L_N$ and $t_l$, and bounding the sum
by an integral, we have
\begin{equation}\label{firstlimit11}
(\ref{firstsum1})  \leq \int_{N^\alpha t_0/s} ^{\infty}
e^{-\beta_1 sl} dl= \beta_1 s e^{-N^\alpha t_0 \beta_1}\,.
\end{equation}
For the second sum on the right-hand side of (\ref{eq::52}), we have that
\begin{multline}\label{secondlimit11}
\sum_{l=\lfloor N\acute{\tau}\rfloor}^{\infty} e^{-t_l \Psi(x,N^{-1} t_l)} \leq \sum_{l=\lfloor N\acute{\tau}\rfloor}^{\infty} e^{-N\beta_2 \log(N^{-1} t_l)}
=\sum_{l=\lfloor N\acute{\tau}\rfloor}^{\infty} \left(\frac{ls}{N} \right)^{-\beta_2 N} \\
\leq s^{-\beta_2 N} N\int_{\acute{\tau}}^{\infty} x^{-\beta_2 N} dx 
= -\frac{ s^{-\beta_2 N} N \acute{\tau} ^{1-\beta_2 N}}{1-\beta_2 N}\,.
\end{multline}
Using (\ref{firstsum1}) and the bounds provided by
(\ref{firstlimit11}) and (\ref{secondlimit11}), we get
\begin{multline*}
\limsup_{N \rightarrow \infty} \frac{1}{N^\alpha} \log
\mathbb{P}\left(\sup_{t>(N^\alpha t_0)/N} [A^{\oplus N}(0,t)-Nxt]\geq
N^\alpha B\right)\\
 \leq \max\left\{-\beta_1 t_0,\limsup_{N \rightarrow \infty} \frac{1}{N^\alpha} \log
\left( \frac{ s^{-\beta_2 N} N \acute{\tau} ^{1-\beta_2 N}}{\beta_2
N-1}\right) \right\}.
\end{multline*}
The second term is $-\infty$, as required.
\end{proof}

\begin{lemma}
Suppose $A$ satisfies Assumptions \ref{ass::ldbg1}, and that $0<\alpha<1$ and $\beta>1$. Then, for any $B, d >0$, there exists $T = T(B, d) \in \mathbb{R}_+$ and $N = N(B, d) \in \mathbb{Z}_+$ such that, for any $t >T$, 
\begin{align*}
\limsup_{N \to \infty} \frac{1}{N^\alpha \log N} \log \mathbb{P} \left( \sup_{t>N^{\alpha - \beta}T } A^{\oplus N} (0,t) - N \lambda t - N^\beta C t \geq N^\alpha B \right) < -d\,.
\end{align*}
\label{lemma: time-bounding for beta > 1}
\end{lemma}

\begin{proof}
We begin by noting that 
\[
A^{\oplus N} (0,t) - N \lambda t - N^\beta C t \geq N^\alpha B  
\implies \frac{A^{\oplus N} (0,t)}{N^\alpha} - N^{1-\alpha} \lambda t - N^{\beta - \alpha} C t \geq B\,.
\]
Let $t' = N^{\beta - \alpha}t$, so that
\begin{align*}
\frac{A^{\oplus N} (0,t)}{N^\alpha} - N^{1-\alpha} \lambda t - N^{\beta - \alpha} C t =\frac{A^{\oplus N} (0,N^{\alpha-\beta}t')}{N^\alpha} - N^{1-\beta} \lambda t' -  Ct'\,. 
\end{align*}
The problem then reduces to bounding
\begin{align}\label{original_inequality}
\mathbb{P} \left(\sup_{t \geq t_0} \left[\tilde{A}^{N}_{\alpha, \beta}(0,t) - Ct \right] \geq B \right)\,.
\end{align}
We choose $\delta \in (0, B)$ such that $s = \delta C^{-1} < 1$. For $l \in \mathbb{Z}^+$, let $t_l = sl$ and $I_l = [t_l, t_{l+1})$. For fixed $N$, $\tilde{A}^{N}_{\alpha, \beta}(0,t)$ is non-decreasing in $t$, and $B - \delta >0$, therefore we get
\begin{align*}
\mathbb{P} \left(\sup_{t \in [t_l, t_{l+1})} \left[\tilde{A}^{N}_{\alpha, \beta}(0,t) - Ct \right] \geq B \right) & \leq \mathbb{P} \left(\tilde{A}^{N}_{\alpha, \beta}(0, t_{l+1}) > B + C t_{l} \right) \\
& \leq \mathbb{P} \left(\tilde{A}^{N}_{\alpha, \beta}(0, t_{l+1}) > B - \delta + Ct_{l+1} \right) \\															& \leq \mathbb{P} \left(\tilde{A}^{N}_{\alpha, \beta}(0, t_{l+1}) > C t_{l+1} \right)\,.
\end{align*}
We use the Chernoff bound to see that, for $\theta>0$, 
\begin{multline*}
\mathbb{P} \left(\tilde{A}^{N}_{\alpha, \beta}(0,t_l) > C t_l \right)  \leq e^{-t_l C \theta} \mathbb{E} \left[e^{\theta \tilde{A}^{N}_{\alpha, \beta}(0, t_l)} \right]\\ 
 = e^{-t_l \left( C \theta + \lambda \theta N^{1-\beta} - N/t_l \Lambda_{\alpha-\beta}(\theta/N^\alpha) \right)} 
															 \leq e^{-t_l \left( C \theta + \lambda \theta N^{1-\beta} - N/t_l \Lambda_{\alpha-\beta}(\theta) \right)} \,.
\end{multline*}
Taking the infimum over $\theta \geq 0$ we get
\begin{align*}
\mathbb{P}\left(\tilde{A}^{N}_{\alpha, \beta}(0,t_l) > Ct_l\right) \leq e^{-t_l\Psi(C+\lambda N^{1- \beta}, N^{\alpha-\beta}t_l)}\,,
\end{align*}
where the last equality is obtained by recognizing that $N/t_l = \frac{N^{\alpha-\beta+1}}{N^{\alpha-\beta}t_l}$, and that for $m \in [1, \infty)$, $$\sup_{\theta \in [0, \infty)} \left[\theta x - \frac{m}{t} \log \mathbb{E}e^{\theta t}\right] \leq \Psi(x,t)\,.$$

Given $t_0 \in (s^{-1}, \infty)$, let $L_N \in \mathbb{Z}^+$ be such that $N^{\alpha-\beta}t_0 \in [t_{L_N}, t_{L_{N+1}})$.  Let $N_0 = \inf\{N:N_\tau > L_N\}$. For $\tau \in (0, \infty)$ and all $N>N_0$  we have that (\ref{original_inequality}) is at most 
\begin{multline}
\P \left( \sup_{t \in [sL_N, s \floor{N\tau})} \left[ \tilde{A}^{N}_{\alpha, \beta}(0,t) - Ct \right] \geq B \right)
 + \P \left(\sup_{t \in [s\floor{N\tau}, \infty)} \left[ \tilde{A}^{N}_{\alpha, \beta}(0,t) - Ct \right] \geq B \right) \\ \leq \sum_{L_N + 1}^{\floor{N \tau}} e^{-t_l \Psi(C+\lambda N^{1-\beta}, N^{\alpha-\beta}t_l)} + \sum_{l = \floor{N \tau}}^{\infty} e^{-t_l \Psi(C+ \lambda N^{1-\beta}, N^{\alpha - \beta}t_l)}\,.
\label{eq::53}
\end{multline}
We proceed similarly to the proof of Lemma \ref{lem::lt2}.  We bound $\Psi(x,t)$ using extensions to Lemma 5 and Corollary 6 in \cite{CaoRamanan2002}: there exist $\dot{\tau} > 1$, and $\beta_1, \beta_2 >0$ such that $\Psi(x,t) \geq \beta_1$ for $t \in [0, \dot{\tau})$ and $t \Psi(x,t)/\log t \geq \beta_2$ for $t \in [\dot{\tau}, \infty)$. This gives
\begin{align*}
\sum_{L_N + 1}^{\floor{N \dot{\tau}}} e^{-t_l \Psi(C+ \lambda N^{1-\beta}, N^{\alpha-\beta}t_l)} \leq \sum_{l=L_N +1}^{\floor{N \dot{\tau}}} e^{-t_l \beta_1} \leq \sum_{l=L_N +1}^{\infty} e^{-t_l \beta_1}\,,
\end{align*}
which tends to zero as $L_N\rightarrow\infty$.
For the second sum on the right-hand side of (\ref{eq::53}) we have that
\begin{multline*}
\sum_{l = \floor{N \dot{\tau}}}^{\infty} e^{-t_l \Psi(C+ \lambda N^{1-\beta}, N^{\alpha - \beta}t_l)}  \leq \sum_{l = \floor{N \dot{\tau}}}^{\infty} e^{-\beta_2 N^{\beta-\alpha} \log(N^{\alpha - \beta t_l})}\\ 
 = \sum_{l = \floor{N \dot{\tau}}}^{\infty} \left(N^{\alpha-\beta} sl \right)^{-\beta_2 N^{\beta - \alpha}} 								 
 = s^{-\beta_2 N^{\beta-\alpha}} N \int_{\dot{\tau}}^{\infty} x^{-\beta_2 N^{\beta-\alpha}} \, dx 
 = - \frac{s^{-\beta_2 N^{\beta - \alpha + 1}} \dot{\tau}^{1 - \beta_2 N^{\beta - \alpha}}}{1 - \beta_2 N^{\beta - \alpha}}\,,
\end{multline*}
which also tends to zero as $N\rightarrow\infty$. We can therefore choose $L_N = L_N(,B)$ and $N =  N(,B)$ such that for $T = sL$ and $n>N$, each term on the right hand side of (\ref{eq::53}) is less than $\epsilon/2$. Hence, (\ref{original_inequality}) is at most $\epsilon$ for sufficiently large $N$, and the required conclusion follows from taking logarithms.
\end{proof}

\subsubsection{Lemmas for moderate deviations (\texorpdfstring{$\beta<1$)}{}}

We now need to provide equivalent lemmas for the moderate deviations scalings. We begin with a limit on the log moment generating function.

\begin{lemma} \label{lem::pple3}
Suppose $A$ is  a marked point process which satisfies Assumptions \ref{ass::mdlb} or \ref{ass::mdsb}. Given $0=j_0<j_1<j_2< \cdots <j_{n-1}<j_n = 1$ and $t \in \mathbb{R}$,  define the vector $\mathbf{A}^{\mathbf{j},t}$ by $\mathbf{A}^{\mathbf{j},t}_i = A(t j_{i-1}, t j_i)$. Also, let $f: \mathbb{R}^+ \rightarrow \mathbb{R}^+$ be a continuous function for which $f(t)\rightarrow0$ as $t\rightarrow0$.  Then there exists $\mathbf{\theta}^* >0$ such that, uniformly for all $\theta\in [0,\theta^*]^n,$
\[
\lim_{t \rightarrow 0} \frac{1}{t (f(t))^2} \log \mathbb{E}\left[ \exp \left \{\langle\theta f(t), \mathbf{A}^{\mathbf{j},t}\rangle - \lambda \mathbb{E}(Y)f(t) \sum_{i=1}^n\theta_i (j_{i}-j_{i-1})t \right \} \right]
 =\frac{1}{2}\lambda \mathbb{E}(Y^2) \sum_{i=1}^n (j_{i}-j_{i-1})\theta_i^2\,.
\]
\end{lemma}

\begin{proof}
As in the proof for Lemma \ref{lem::pple} we know $\mathbb{P}(X(0,t)=1) =\lambda t +o(t)$ and
$\mathbb{P}(X(0,t) \geq 2) =o(t)$. In addition, by assumption, using (\ref{eq::explimit}) and making use of the expansion of $e^x$ about $x=0$ we get
\begin{equation}
\label{eq:qwerty}\mathbb{E}\left[ e^{\langle\theta f(t), \mathbf{A}^{\mathbf{j},t}\rangle} \right] -1= \mathbb{E}\left[ \sum_{i=1}^{\infty}  (\langle\theta f(t), \mathbf{A}^{\mathbf{j},t}\rangle)^i/i!\right]\,.
\end{equation}
Now we  consider $\mathbb{E}(\langle\theta f(t), \mathbf{A}^{\mathbf{j},t}\rangle^i)$ for $i>0$.  This is bounded above by $\mathbb{E}((\theta^* f(t) A(0,t))^i)$, as $\max \theta_i \leq \theta^*$. Using (\ref{eq:qwerty}) and exchanging the sum and expectation, we have, for $\nu< \theta^*$,
$$ \sum_{i=1}^{\infty} \nu^i \mathbb{E}( A(0,t)^i)/i! = \lambda t( M(\nu)-1) +o(t)\,.$$
Since $A(0,t)\geq 0$ we have $\mathbb{E}( A(0,t)^i) = c_i t +o(t)$ for any $i>0$, giving  $\mathbb{E}((\theta^* f(t) A(0,t))^i) = c_i (f(t)\theta^*)^i t + o(t f(t)^i)$, which will be used to bound $\mathbb{E}(\langle\theta f(t), \mathbf{A}^{\mathbf{j},t}\rangle^i)$. 
We know $$\mathbb{E}(\langle\theta f(t), \mathbf{A}^{\mathbf{j},t}\rangle) = \lambda \mathbb{E}(Y)f(t) \sum_{i=1}^n\theta_i (j_{i}-j_{i-1})t\,.$$
We now let $p_k(t) = \mathbb{P}(X(0,t)=k)$, so get 
\[
\mathbb{E}(\langle\theta f(t), \mathbf{A}^{\mathbf{j},t}\rangle^2)\\
 = \sum_{i=1}^k p_i(t) \mathbb{E}(\langle\theta f(t), \mathbf{A}^{\mathbf{j},t}\rangle^2| X(0,t)=i) + \mathbb{E}(\langle\theta f(t), \mathbf{A}^{\mathbf{j},t}\rangle^2 {1}_{X(0,t)>K})\,.
\]
Since we have $\mathbb{P}(X(0,t)>2) = o(t)$ we get $p_i(t) \mathbb{E}(\langle\theta f(t), \mathbf{A}^{\mathbf{j},t}\rangle^2| X(0,t)=i) = o(t f(t)^2)$ for $i>1$. 
Putting this all together we get
\begin{multline*}
\mathbb{E}\left[ e^{\langle\theta f(t), \mathbf{A}^{\mathbf{j},t}\rangle} \right] -1 =
\lambda \mathbb{E}(Y)f(t) \sum_{i=1}^n\theta_i (j_{i}-j_{i-1})t \\
+ \frac{1}{2}(\lambda t +o(t)) \mathbb{E}(\langle\theta f(t), \mathbf{A}^{\mathbf{j},t}\rangle^2| X(0,t)=1)+ o(t f(t)^2) 
\\+ \frac{1}{2}\mathbb{E}(\langle\theta f(t), \mathbf{A}^{\mathbf{j},t}\rangle^2 {1}_{X(0,t)>K}) 
+\sum_{i=3}^{\infty} \mathbb{E}(\langle\theta f(t), \mathbf{A}^{\mathbf{j},t}\rangle^i)\,.
\end{multline*}
Using (\ref{eqn::shortprob}) we get
 $$\mathbb{E}(\langle\theta f(t), \mathbf{A}^{\mathbf{j},t}\rangle^2| X(0,t)=1) = \frac{1}{2}\mathbb{E}(Y^2) \sum_{i=1}^n (j_{i}-j_{i-1})(\theta_i f(t))^2\,.$$
As $x-x^2/2 -c \leq \log(1+x) -c \leq x-c$, we use the above to get 
\begin{multline} 
\lim_{t \rightarrow 0}\frac{1}{tf(t)^2}\left( \mathbb{E}\left[ e^{\langle\theta f(t), \mathbf{A}^{\mathbf{j},t}\rangle} \right] -1 -\lambda \mathbb{E}(Y)f(t) \sum_{i=1}^n\theta_i (j_{i}-j_{i-1})t \right )\\=  \frac{1}{2}\lambda \mathbb{E}(Y^2) \sum_{i=1}^n (j_{i}-j_{i-1})\theta_i^2+ \lim_{t \rightarrow 0}\frac{1}{tf(t)^2}\mathbb{E}(\langle\theta f(t), \mathbf{A}^{\mathbf{j},t}\rangle^2 {1}_{X(0,t)>K})\,.
\end{multline}
The second term goes to zero since, by assumption, we have 
$$
\lim_{t\rightarrow 0}t^{-1}
\mathbb{E}\left[e^{\theta_0 A(0,t)}1_{\{X(0,t)>K\}}\right]=0\,.
$$
This gives the desired result.
\end{proof}

\begin{lemma}\label{lem::lt3}
Let $A$ be a marked point process such that $\Psi_{\infty,d\rightarrow t}^2(x) > 0$  and  $\Psi_{\infty,t\rightarrow d}^2(x) > 0$. Given $B>0$, $d >0$ and $\alpha, \beta >0$ such that $\beta <1$ and $\alpha +\beta>1$, there exists $t_0 <\infty$ such that
\begin{multline*}
\limsup_{N \rightarrow \infty} \frac{1}{N^{\alpha+\beta -1}} \log
\mathbb{P}\left(\sup_{t>N^{\alpha-\beta} t_0} [A^{\oplus N}(0,t)-N \lambda\mathbb{E}(Y) t- N^{\beta}x t]\geq N^\alpha B \right)\\
\leq -d.
\end{multline*}
\end{lemma}

\begin{proof}
We use the scaling $F^{\oplus N}(0,t) =A^{\oplus N}(0,t/N)$,  giving
\begin{multline}
\mathbb{P}\left(\sup_{t>N^{\alpha-\beta} t_0}  [A^{\oplus N}(0,t)-N \lambda\mathbb{E}(Y) t- N^{\beta}x t]\geq N^\alpha B \right) \\ =\mathbb{P}\left(\sup_{t>N^{\alpha-\beta +1} t_0} [F^{\oplus N}(0,t)-\lambda\mathbb{E}(Y) t- N^{\beta-1}x t]\geq N^{\alpha}B \right)\,.
\label{eqn::probscaling}
\end{multline}
We choose $N_1$ such that $N_1^\alpha B > \lambda \mathbb{E}(Y)+N_1^{\beta -1}x$, then choose $\delta \in (0,N_1^\alpha B)$ such that $\delta > \lambda \mathbb{E}(Y)+N_1^{\beta -1}x$. Note that for $N>N_1$ we have $N^\alpha B > N_1^\alpha B$ and $\lambda \mathbb{E}(Y)+N^{\beta -1}x <\lambda \mathbb{E}(Y)+N_1^{\beta -1}x$.

For $l \in \mathbb{Z}^+$, let $I_l =
[l, {l+1})$. Since $F^{\oplus N}(0,t)$ is non-decreasing
in $t$, and because we have $N^\alpha B-\delta >0$ for $N>N_1$, we get
\begin{align*}
\mathbb{P}&\left(\sup_{t\in [l , {l+1} ) } [F^{\oplus N}(0,t)-\lambda\mathbb{E}(Y) t- N^{\beta-1}x t]\geq N^\alpha B\right)\\ 
&\leq \mathbb{P}(F^{\oplus N}(0,{l+1})> N^\alpha B +(\lambda\mathbb{E}(Y) + N^{\beta-1}x) l)\\
&\leq \mathbb{P}(F^{\oplus N}(0,{l+1})> N^\alpha B - \delta +(\lambda\mathbb{E}(Y) + N^{\beta-1}x) (l+1))\\
&\leq \mathbb{P}(F^{\oplus N}(0,{l+1})> (\lambda\mathbb{E}(Y) + N^{\beta-1}x) (l+1))\,.
\end{align*}
For the remainder of this proof we will assume $N> N_1$.
We can now use the Chernoff bound which, for $\theta\geq 0$, gives
\begin{multline*}
\mathbb{P}(F^{\oplus N}(0,{l})> (\lambda\mathbb{E}(Y) + N^{\beta-1}x) l) \leq e^{-(\lambda\mathbb{E}(Y) + N^{\beta-1}x) l \theta} \mathbb{E}[e^{\theta F^{\oplus N}(0,{l})}]\\
=e^{-l \left( (\lambda\mathbb{E}(Y) + N^{\beta-1}x) \theta -N l^{-1} \log \mathbb{E}[e^{\theta A(0,{l}/N)}]\right) }\,.
\end{multline*}
Taking the infimum over $\theta \geq 0$ we get
$$
\mathbb{P}(F^{\oplus N}(0,{l})>(\lambda\mathbb{E}(Y) + N^{\beta-1}x){l}) \leq e^{-l \Psi((\lambda\mathbb{E}(Y) + N^{\beta-1}x),N^{-1} l)}\,.
$$

Given $t_0\in (1, \infty)$, let $L_N \in \mathbb{Z}^+$ be such
that $N^{\alpha- \beta +1} t_0 \in [{L_N} ,{L_N +1})$.  Let $N_2 = \inf \{ N: N
\tau >L_N \}$.  For some
$\tau \in (0, \infty)$ and all $N>\max (N_1,N_2)$  we find that (\ref{eqn::probscaling}) is at most
\begin{multline}
 \mathbb{P} \left(\sup_{t\in [L_N ,  \lfloor N\tau\rfloor ) } [F^{\oplus N}(0,t)-(\lambda\mathbb{E}(Y) + N^{\beta-1}x)t]\geq N^\alpha B\right)
 \\ +\mathbb{P} \left(\sup_{t\in [ \lfloor N\tau\rfloor , \infty ) } [F^{\oplus N}(0,t)-(\lambda\mathbb{E}(Y) + N^{\beta-1}x)t]\geq N^\alpha B\right) \\
\leq \sum_{l=L_N+1}^{\lfloor N\tau\rfloor} e^{-l \Psi((\lambda\mathbb{E}(Y) + N^{\beta-1}x),N^{-1} l)} 
+ \sum_{l=\lfloor N\tau\rfloor}^{\infty} e^{-l \Psi((\lambda\mathbb{E}(Y) + N^{\beta-1}x),N^{-1} l)} \,.
\label{eq::55}
\end{multline}
We consider $\Psi(\lambda\mathbb{E}(Y) + x,t)$ to find bounds on the above. We start with the assumption that
$\Psi_{\infty,d\rightarrow t}^2(x) > 0 \textrm{ and  } \Psi_{\infty,t\rightarrow d}^2(x) > 0$.  There exist $D$ and $T_1$ such that, for all $x<D$ and $t>T_1$, 
$$\frac{t\Psi(\lambda\mathbb{E}(Y) + x,t)}{x^2 \log(t)} > \delta_1\,.$$
 Also, we have
 \begin{multline}
  \Psi(\lambda\mathbb{E}(Y) + N^{\beta-1} x,t)  = \sup_{\theta} \left\{\theta(\lambda\mathbb{E}(Y) + N^{\beta-1}x) -t^{-1}\log\mathbb{E}( e^{A(0,t)\theta})\right\}\\
  = N^{2(\beta -1)} \sup_{\theta} \left\{\theta N^{1-\beta}x -(tN^{2(\beta -1)})^{-1}\log\mathbb{E}( e^{A(0,t)\theta -\theta\lambda\mathbb{E}(Y)})\right\}.  \label{eqn::thing}
 \end{multline}
 Letting $\theta' = N^{1-\beta} \theta$,
 \begin{align*}
  (\ref{eqn::thing}) &= N^{2(\beta -1)} \sup_{\theta} \left\{\theta x -(tN^{2(\beta -1)})^{-1}\log\mathbb{E}( e^{A(0,t) N^{\beta -1}\theta -N^{\beta -1}\theta\lambda\mathbb{E}(Y)})\right\}\,.\notag
\end{align*}
From Lemma \ref{lem::pple3} there is $\theta_0$ such that, for all $\theta \in (0, \theta_0)$ and however we take the joint limit, we get 
$$
\lim_{N\rightarrow \infty, t\rightarrow 0} \frac{1}{tN^{2(\beta -1)}}\log\mathbb{E}( e^{A(0,t) N^{\beta -1}\theta -N^{\beta -1}\theta\lambda\mathbb{E}(Y)}) = \lambda \mathbb{E}(Y^2) \theta^2\,.
$$ 
So we have that, for $x \neq 0$, 
$$
\lim_{N\rightarrow \infty, t\rightarrow 0} (N^{2(1-\beta)}) \Psi(\lambda\mathbb{E}(Y) + N^{\beta-1} x,t) >0\,.
$$
This means that there is $T_2<1$ and $N_3$ such that, for all $N>N_3$, $(N^{2(1-\beta)}) \Psi(\lambda\mathbb{E}(Y) + N^{\beta-1} x,t) > \delta_2$.  Using H\"older's inequality and stationarity of $A$ means that for $t\in(T_2, \tau)$ we have $\Psi(\lambda\mathbb{E}(Y) + N^{\beta-1} d,t)  > N^{2(\beta-1)} (\delta_2 T_2/t)$.

Let $\acute{\tau} =\max(e^{\frac{d+1}{\delta_1}}, T_1)$ and $N>\max(N_1, N_2, N_3)$. Using the above results we bound the terms of interest: 
\begin{equation}\label{firstsum}
 \sum_{l=L_N+1}^{\lfloor N\acute{\tau}\rfloor} e^{-l \Psi((\lambda\mathbb{E}(Y) + N^{\beta-1}x),N^{-1} l)}  \leq \sum_{l=L_N+1}^{\lfloor N\acute{\tau}\rfloor} e^{-l N^{2(\beta-1)} (\delta_2 T_2/\acute{\tau})}\\
  \leq \sum_{l=L_N+1}^{\infty} e^{-l N^{2(\beta-1)} (\delta_2 T_2/\acute{\tau})}\,.
\end{equation}
Using the definition of $L_N$  and bounding the sum
by an integral we have
\begin{equation}\label{firstlimit1}
\sum_{l=L_N+1}^{\infty} e^{-l N^{2(\beta-1)} (\delta_2 T_2/\acute{\tau})} \leq \int_{N^{\alpha -\beta +1} t_0} ^{\infty}
e^{-l N^{2(\beta-1)} (\delta_2 T_2/\acute{\tau})} dl
= N^{2(\beta-1)} (\delta_2 T_2/\acute{\tau})e^{- t_0 N^{\alpha+\beta-1} (\delta_2 T_2/\acute{\tau})}\,.
\end{equation}
For the second sum in (\ref{eq::55}) we have
\begin{multline}
\sum_{l=\lfloor N\acute{\tau}\rfloor}^{\infty} e^{-l \Psi((\lambda\mathbb{E}(Y) + N^{\beta-1}x),N^{-1} l)} \leq \sum_{l=\lfloor N\acute{\tau}\rfloor}^{\infty} e^{-N^{2\beta-1}\delta_1 \log(N^{-1} l)}\\ 
=\sum_{l=\lfloor N\acute{\tau}\rfloor}^{\infty} \left(\frac{l}{N} \right)^{-\delta_1 N^{2\beta -1}} 
\leq  N\int_{\acute{\tau}}^{\infty} x^{-\delta_1 N^{2\beta -1}} dx 
= -\frac{ N \acute{\tau} ^{1-\delta_1 N^{2\beta -1}}}{1-\delta_1 N^{2\beta -1}}\,.
\label{secondlimit1}
\end{multline}

So, using (\ref{firstsum}) and then the bounds provided by
(\ref{firstlimit1}) and (\ref{secondlimit1}) we get that
\[
\limsup_{N \rightarrow \infty} \frac{1}{N^{\alpha+\beta -1}} \log
\mathbb{P}\Bigg(\sup_{t>(N^{\alpha-\beta +1} t_0)} [A^{\oplus N}(0,t)
-(N\lambda \mathbb{E}(Y) +N^\beta)t]\geq
N^\alpha B\Bigg)
\]
is bounded above by
$$
\max\left\{-t_0\delta_2 T_2/\acute{\tau} ,\limsup_{N \rightarrow \infty} \frac{1}{N^{\alpha+\beta-1}} \log
\left( \frac{ N \acute{\tau} ^{\delta_1 N^{2\beta -1}}}{1-\delta_1 N^{2\beta -1}-1}\right) \right\}\,.
$$
The final term is $-\infty$ for $\alpha <\beta$, and bounded above by $-d$, showing the desired result.
\end{proof}

\subsubsection{Properties of the processes}

We now look more closely at the processes of interest. We begin by checking that these processes belong to the space $\mathcal{D}_0$.

\begin{lemma}\label{lem::C0}
Let $A$ be a marked point process such that $\Psi^1_\infty(x) >0$ for all $x \neq \lambda \mathbb{E}(Y)$. Then $\tilde{A}^N_{\alpha,\beta}$ is in $\mathcal{D}_{0}$ for all $N$.
\end{lemma}
\begin{proof}
Firstly, by definition we have that $\tilde{A}^N_{\alpha,\beta}$ is a cadlag function. Secondly, we need to show that $\tilde{A}^N_{\alpha,\beta} /(t+1) \rightarrow 0$ as $t \rightarrow \infty$. This is equivalent to showing
$ \mathbb{P}(\tilde{A}^N_{\alpha,\beta} \notin \mathcal{D}_{0}) =0$. 
We do this by considering, for $x>\lambda \mathbb{E}(Y)$, $\mathbb{P}(A(0,t)>xt)\leq e^{-t \Psi(x,t)}$.  Since $\Psi^1_\infty (x) >0$ given  $x \neq \lambda \mathbb{E}(Y)$,  there exist $T$ and $\delta>0$ such that, for all $t>T$, $t \Psi(x,t)/\log t \geq \delta > 0$.
This shows that  $\mathbb{P}(A(0,t)>xt) \rightarrow 0$ as $t$ tends to $\infty$. A similar argument holds for $x<\lambda \mathbb{E}(Y)$.
\end{proof}

For our Definition \ref{def:linerinterp} of the linear interpolation of a point process to be of use, we need to show the processes are exponentially tight (\cite[p.8]{Dembo1998}) on the scalings of interest. 
\begin{lemma}\label{poly}
Let $\alpha>0$ and $A$ be a marked point process. If $\beta \geq 1 $ and
$M(\theta) < \infty$ for all $\theta \in \mathbb{R}$, or if $\beta <1$ and 
there exists $\theta>0$ such that $M(\theta) <\infty$, then the associated measures of $\tilde{A}^N_{\alpha ,\beta}$ (the scaled original process) and $\tilde{\bar{A}}^N_{\alpha,\beta}$ (the scaled linear interpolation) are exponentially equivalent in $\mathcal{D}_{0}$ with the scaled uniform norm $||\cdot ||_s$.
\end{lemma}
\begin{proof}
The sets we need to consider are $$\left\{  \left\Vert \tilde{\bar{A}}^N_{\alpha,\beta} -\tilde{A}^N_{\alpha,\beta}\right\Vert_s >\gamma \right\}\,, $$  
which are obviously measurable. To estimate the measure of these we consider whether this event happens in intervals of the form $[i,i+1)$:
$$\mathbb{P}\left( \left\Vert \tilde{\bar{A}}^N_{\alpha,\beta} -\tilde{A}^N_{\alpha,\beta}\right\Vert_s >\gamma \right ) \leq \sum_{i =0}^\infty \mathbb{P} \left( \| \tilde{\bar{A}}^N_{\alpha,\beta} -\tilde{A}^N_{\alpha,\beta}\|_{[i,i+1)} > \gamma (i+1) \right) \,,
$$
where $\|x\|_{[i,i+1)} = \sup_{t\in[i,i+1)}|x(t)|$. 
Now, we have that $$|\tilde{\bar{A}}^N_{\alpha,\beta}(0,t) -\tilde{A}^N_{\alpha,\beta}(0,t)| < |\zeta_N(N^{\alpha -\beta}t)|/ N^{\alpha}\,,$$ which implies that if \mbox{$| \tilde{\bar{A}}^N_{\alpha,\beta} -\tilde{A}^N_{\alpha,\beta} |_{[i,i+1)} >\gamma (i+1)$}, either there is a mark in
$(-N^{\alpha -\beta} (i+1),-N^{\alpha -\beta} i]$ or $\zeta_N(N^{\alpha -\beta} (i+1))$ is larger than $N^\alpha \gamma (i+1)$.  So, for any $\gamma>0$,
\begin{multline}\label{eqn::mnb}
\mathbb{P}(| \tilde{\bar{A}}^N_{\alpha,\beta} -\tilde{A}^N_{\alpha,\beta}|_{[i,i+1)} > \gamma (i+1) )  \\ \leq \sum_{k=1}^{N}\mathbb{P}\left( \bigcup_{j=X^{(k)}(0,N^{\alpha-\beta}(i))}^{X^{(k)}(0,N^{\alpha-\beta}(i+1))} \left\{ Y_j^{(i)} > \gamma (i+1) N^{\alpha} \right\} \right)
+ \mathbb{P}\left(Y_{1}^{(1)} > \gamma (i+1) N^{\alpha}\right)\,,
\end{multline}
where the first term represents marks in $(-N^{\alpha -\beta}(i+1),-N^{\alpha -\beta}i]$, and the second the first mark before this.  Now, the RHS of (\ref{eqn::mnb}) is
\begin{align*}
&N\mathbb{P}\left( \bigcup_{j=X^{(k)}(0,N^{\alpha-\beta}(i))}^{X^{(k)}(0,N^{\alpha-\beta}(i+1))} \{ Y_j^{(1)} > \gamma (i+1) N^{\alpha} \} \right)+\mathbb{P}(Y_{1}^{(1)} > \gamma (i+1) N^{\alpha})\\
&=N \mathbb{E} \left( 1_{\{\cup_{j=X^{(k)}(0,N^{\alpha-\beta}(i))}^{X^{(k)}(0,N^{\alpha-\beta}(i+1))} \{ Y_j^{(1)} > \gamma (i+1) N^{\alpha} \} \}} \right) +\mathbb{P}(Y_{1}^{(1)} > \gamma (i+1) N^{\alpha})\\
&\leq N \mathbb{E} \left(X^{(1)}(N^{\alpha-\beta}i,N^{\alpha-\beta}(i+1)) 1_{ \{ Y_1^{(1)} > \gamma (i+1) N^{\alpha} \}} \right)+\mathbb{P}(Y_{1}^{(1)} > \gamma (i+1) N^{\alpha}) \\
&=(N^{\alpha -\beta  +1} \lambda +1) \mathbb{P}(Y_{1}^{(1)} > \gamma (i+1) N^{\alpha}) \,.
\end{align*}
We now apply the Chernoff bound to this to get that, for $\theta \geq 0$,
$$
\mathbb{P}(| \tilde{\bar{A}}^N_{\alpha,\beta} -\tilde{A}^N_{\alpha,\beta}|_{[i,i+1)} > \gamma (i+1) ) \leq (N^{\alpha - \beta +1} \lambda +1) M(\theta) e^{-N^\alpha \theta \gamma (i+1)}\,.
$$
Summing the terms and bounding the sum by an integral we get 
\begin{align*}\mathbb{P}\left( \left\Vert \tilde{\bar{A}}^N_{\alpha,\beta} -\tilde{A}^N_{\alpha,\beta}\right\Vert_s >\gamma \right ) &\leq \sum_{i =0}^\infty (N^{\alpha - \beta +1} \lambda +1) M(\theta) e^{-N^\alpha \theta \gamma (i+1)} \\
&\leq e^{-N^\alpha \theta \gamma }((N^{\alpha - \beta +1} \lambda +1) M(\theta))(1+ (N^\alpha \theta \gamma)^{-1})\,.
\end{align*}
Now, for $\beta <1$ we just select $\theta > 0$ such that $M(\theta) < \infty$ and we let $N \rightarrow \infty$.  If $\beta \geq 1$ we know that, since $M(\theta)<\infty$ for all $\theta$, we can take  $N \rightarrow \infty$ and then $\theta \rightarrow \infty$, which gives 
\begin{displaymath}
\limsup_{N\rightarrow\infty} \frac{1}{N^{\alpha + \beta -1}} \log \mathbb{P}\left(\left \Vert \tilde{\bar{A}}^N_{\alpha,\beta}|_{[0,T]} -\tilde{A}^N_{\alpha,\beta}|_{[0,T]} \right\Vert_s>\gamma \right)  = -\infty\,,
\end{displaymath}
for any $\gamma>0$. This shows the two measures are exponentially tight.
\end{proof}

\subsection{Framework for the proof of the large deviations principles}\label{sec::framework}
In this section we provide a series of lemmas which form the basis of the proofs of our main results; the proofs will be completed in subsequent sections.
We start by proving a finite dimensional large deviations principle, which is then extended into a sample path result in finite time.
These lemmas are based on ideas from the proof of Mogulskii's theorem \cite[Theorem 5.1.2]{Dembo1998} and the extension of this to infinite time by O'Connell and Ganesh \cite{Ganesh2002}.  

\begin{lemma} \label{lem::finiteA}
Let $\mathscr{J}$ be the collection of all ordered finite subsets
of $(0,1]$. For any  $T \in \mathbb{R}^+$, $j =\{
0<s_1<s_2<\cdots<s_{\lvert j\rvert }\leq 1\} \in \mathscr{J}$ and
$f: [0,T]\times[0,T] \rightarrow \mathbb{R }$, let $p^T_j (f)$ denote
the vector $(f(0,T s_1),f(0,T s_2), \dots, f(0,T s_{\lvert j
\rvert})) \in \mathbb{R}^{\lvert j\rvert}$. Let $A$ be a marked point process. Given $\alpha, \beta \in \mathbb{R}^+$ such that either $0<\alpha<\beta$ and $\alpha+\beta>1$, or $0<\alpha<1$ and $\beta>1$, we define
\[
  \Omega_{\alpha, \beta}(\theta, j,T) = \lim_{N\rightarrow \infty} \frac{1}{f(N)}\log \mathbb{E} \left( e^{\sum_{i=1}^{\lvert j \rvert} N^{\beta -1} \theta_i A( N^{\alpha -\beta} j_{i-1} T, N^{\alpha -\beta} j_{i} T )- \theta_i T(j_i -j_{i-1})N^{\alpha -1}\lambda\mathbb{E}(Y)}\right),  
\]
where
\begin{align*}
    f(N) = \left \lbrace \begin{array}{ll}
       N^{\alpha + \beta - 2}  & \text{ if } 0< \alpha < \beta \text{ and } \alpha + \beta > 1\,,  \\
       N^{\alpha} \log N  &  \text{ if } 0< \alpha < 1 \text{ and } \beta > 1\,.
    \end{array} \right.
\end{align*}
We assume the following: 
\begin{enumerate}
\item $\Omega_{\alpha, \beta}(\theta, j,T)$ is steep \cite[Definition 2.3.5]{Dembo1998} in the variable $\theta$ for every $j$ and $T$,
\item $\mathcal{D}^\circ_{\Omega_{\alpha, \beta}}$  contains $\theta =0$,
\item $\Omega_{\alpha, \beta}(\theta, j,T)$ is differentiable in  $\theta$ on  $\mathcal{D}^{\circ}_{\Omega_{\alpha, \beta}}$.
\end{enumerate} 
The sequence of vectors $p^T_j\left ( \tilde{A}^N_{\alpha, \beta} \right)$ satisfies
 a large deviations principle in  $\mathbb{R}^{\lvert j \rvert}$ with rate
$f(N)$ and good rate function
\begin{equation*}
I_{j,T}(\mathbf{z})=
\sup_{\theta \in \mathbb{R}^{\lvert j \rvert} } \left\{ \sum_{i=1}^{\lvert j \rvert} (z_i-z_{i-1}) \theta_i - \Omega_{\alpha, \beta}(\theta, j,T) \right \}\,.
\end{equation*}
\end{lemma}

\begin{proof}
We have that, for a fixed $N$, $p_j^T\bigl( \tilde{A}_N \bigr)$ is the following random vector:
\begin{displaymath}
\mathbf{Z}_N^j =( \tilde{A}^N_{\alpha, \beta} (0,T s_1), \dots ,  \tilde{A}^N_{\alpha,\beta}(0,T s_{\lvert j \rvert}))\,.
\end{displaymath}
Let
\begin{displaymath}
\mathbf{W}_N^j =( \tilde{A}^N_{\alpha,\beta}(0,T s_1), \tilde{A}^N_{\alpha,\beta}(T s_1,T s_2),\dots ,\tilde{A}^N_{\alpha,\beta}(T s_{\lvert j \rvert-1} , T s_{\lvert j \rvert}))\,.
\end{displaymath}
We prove a large deviations principle for $\mathbf{W}$  and  apply the contraction mapping principle to find a large deviations principle for $\mathbf{Z}$. We use the
G\"artner-Ellis theorem \cite[Theorem 2.3.6]{Dembo1998}, employing the log moment generating function of $W_N^j$,
$\Psi_{N} (\theta) = \log \mathbb{E}[ e^{\langle\theta, W_N^j \rangle}]$. We
have
\[
 \Psi_{N} (\theta) 
 =N\log \mathbb{E} \Bigg( \exp \Bigg(\sum_{i=1}^{\lvert j \rvert} \frac{1}{N^\alpha} \theta_i A( N^{\alpha -\beta} j_{i-1} T, N^{\alpha -\beta} j_{i} T )
  - \theta_i T(j_i -j_{i-1})N^{-\beta}\lambda\mathbb{E}(Y)\Bigg) \Bigg)\,.
\]
By definition $\Psi(\theta) = \lim_{N\rightarrow \infty} \frac{1}{f(N)} \Psi_N( \theta f(N) )=\Omega_{\alpha, \beta}(\theta, j,T)$.
We need $\Psi$ to satisfy the conditions of the G\"artner-Ellis
theorem; this is ensured by the assumptions on $\Omega_{\alpha, \beta}$.
So the rate function for the large deviations principle for $\mathbf{W}_N^j$ is
\begin{align*}
\Psi^* (w) = \sup_{ \theta \in \mathbb{R}^{\lvert j \rvert} } \left\{ \langle w,\theta \rangle - \Psi (\theta) \right\}
= \sup_{\theta \in \mathbb{R}^{\lvert j \rvert} } \left\{ \sum_{i=1}^{\lvert j \rvert} w_i \theta_i - \Omega_{\alpha, \beta}(\theta, j,T) \right \}\,.
\end{align*}
Since the map $\mathbf{W}_N^j \mapsto \mathbf{Z}_N^j$  is continuous and one
to one, we can find the large deviations principle for $\mathbf{Z}_N^j$ by the
contraction principle \cite[Theorem 4.2.1]{Dembo1998}. 
\end{proof}

As $\tilde{A}^N_{\alpha,\beta}$ and $\tilde{\bar{A}}^N_{\alpha,\beta}$ are exponentially equivalent in the cases considered above, the results of Lemma \ref{lem::finiteA}  hold for $\tilde{\bar{A}}^N_{\alpha,\beta}$.

Next we have a lemma which allows us to turn the previous finite dimensional result into a sample path result.

\begin{lemma} \label{lem::finitesp}
For a given $T\in \mathbb{R}^+$, let $\mathscr{C}^T$ be the space of continuous functions, $x:[0,T]\mapsto \mathbb{R}$ for which $x(0)=0$. Let $W^N$ be a random variable on $\mathscr{C}^T$ such that, given any $j\in \mathscr{J}$ and $0<T'\leq T$, we have that $p^{T'}_j (W^N)$ obeys a large deviations principle with rate $f(N)$ and good rate function $I_{j,T'}$. Then $W^N$ obeys a large deviations principle in $\mathscr{C}^T$ with the topology of pointwise convergence with rate $f(N)$ and good rate function 
$$I_T(x) = \sup_{j\in \mathscr{J}, 0<T'\leq T} I_{j,T'}(x).$$
Furthermore, if $W^N$ is exponentially tight with respect to the scaled uniform norm then $W^N$ obeys a large deviations principle in $\mathscr{C}^T$ with the scaled uniform norm and the same rate and rate function.
\end{lemma}

\begin{proof}
We firstly define a partial ordering on $\mathscr{J}$ by $i =\left( s_1, \dots , s_{\lvert i \rvert} \right) \leq j =\left( t_1, \dots , t_{\lvert j \rvert} \right )$, $i,j
\in \mathscr{J}$, if and only if for each $l$ there exists $q(l)$ such that $s_l =t_{q(l)}$. For a fixed $T$ we can define the projection
$p^T_{ij} : \mathbb{R}^{\lvert j \rvert} \rightarrow \mathbb{R}^{\lvert i \rvert}$
for $i\leq j \in \mathscr{J}$ in the natural way. We now define
$\tilde{\mathscr{C}}^T$ to be the projective limit of $\left\{
\mathscr{Y}_j = \mathbb{R}^{\lvert j \rvert} \right \}_{j\in
\mathscr{J}}$ with respect to the projections $p_{ij}^T$ for a
fixed $T$. The spaces $\tilde{\mathscr{C}}^T$ and $\mathscr{C}^T$
may be identified with each other. This can be seen by considering
$f \in \mathscr{C}^T$, which then corresponds to $(p^T_j (f))_{j\in
\mathscr{J}}$ which belongs to $\tilde{\mathscr{C}}^T$ since
$p^T_i(f)=p_{ij}^T(p_j^T(f))$ for $i\leq j \in \mathscr{J}$. Also, for
$\mathbf{x} \in  \tilde{\mathscr{C}}^T$ we can see that this corresponds to $f \in
\mathscr{C}^T$, where $f(t)= x_k$ if $t \in [s_k, s_{k+1})$ for $t>0$ and $f(0)=0$. In
addition, the projective topology on $\tilde{\mathscr{C}}^T$ is
equivalent to the pointwise convergence topology on
$\mathscr{C}^T$. Therefore we can use the Dawson-G\"artner
theorem \cite[Theorem 4.6.1]{Dembo1998} to find a large deviations principle in $\mathscr{C}^T$ with the
topology of pointwise convergence. The good rate function for this
large deviations principle is
\begin{equation*}
I_T(f) =\sup_{j\in \mathscr{J}, 0<T'\leq T} I_{j,T'}(f)\,.
\end{equation*}

For the second part of the lemma we make use of the inverse contraction principle \cite[Theorem 4.2.4]{Dembo1998}, since $W^N$ is exponentially tight in $\mathscr{C}^T$ with the scaled uniform norm. We use the identity function mapping on $\mathscr{C}^T$ from the topology of pointwise convergence to the uniform topology to give the required result. 
\end{proof}

We thus have a large deviations principle for sample paths of $[0,T)$. Now we will extend it to those on $[0,\infty)$. Again, we use the Dawson-G\"artner theorem to carry out the extension. 

\begin{lemma}\label{lem::infinitesp}
Let $W^N$ be a stationary random process on the space $\mathscr{C}$ of continuous functions of $\mathbb{R}^+$. Assume that, given $T>0$, the random variables $W^N |_T$, the restriction of $W^N$ to the interval $[0,T)$, obey a large deviations principle on $\mathscr{C}^T$ with the uniform topology, rate $f(N)$ and good rate function $I_T$. Assume also that, given $x > \mathbb{E}(W^N(0,1))=\lambda$ and $B,d>0$, there exists $t_{d,x}>0$ such that
$$\limsup_{N\rightarrow \infty} \frac{1}{f(N)} \log \mathbb{P} \left( \sup_{t> t_{d,x}} W^N(0,t) -xt > B\right) <-d,$$
and a similar inequality holds for $x< \mathbb{E}(W^N(0,1))$. Then $W^N$ obeys a large deviations principle on $\mathscr{C}$ with topology induced by the scaled uniform norm $|| \cdot ||_s$ with rate $f(N)$ and good rate function $I(x) = \sup_T I_T(x)$.
\end{lemma}

\begin{proof}
We consider the projections $q_{st}: \mathscr{C}^t \rightarrow \mathscr{C}^s$, for $s \leq t \in \mathbb{R}^+$, which are the restrictions of the
functions to the interval $[0,s)$. This means that the projective limit space
is $\mathscr{C}$ with the projective limit topology. So we can
apply the Dawson-G\"artner theorem 
\cite[Theorem 4.6.1]{Dembo1998}, which shows that
$Y^N$ satisfies a large deviations principle in $\mathscr{C}$ with rate $f(N)$ and good rate
function $I(x) =\sup_{T\in \mathbb{R}^+} I_T(x)$.

We now strengthen the topology from the projection topology to that induced by the scaled uniform norm. We start by proving that $W_N$ is exponentially tight in $\mathscr{C}$ with the scaled uniform norm. We will define two groups of sets, $L_\alpha$ and $K_\alpha$. Firstly, we
know that $W^N|_T$ satisfies a large deviations
principle in the space $\mathscr{C}^T$ with the uniform norm, which is a Polish space. Thus, it is exponentially tight in this space,
which means that there is a family of compact sets $L_\alpha^T$
which have the property that
\begin{displaymath}
\limsup_{N\rightarrow \infty} \frac{1}{f(N)} \log \mathbb{P}(W^N|_T \notin L_\alpha^T)<-\alpha\,.
\end{displaymath}
Let $\bar{L}_\alpha^T$ be the smallest compact set with this
property. Then $\bar{L}_\alpha^{T_2} \subset \bar{L}_\alpha^{T_1}$
if $T_2>T_1$. To see this, assume for contradiction that there is a
function $x$ which is in $\bar{L}_\alpha^{T_2}$ but not in
$\bar{L}_\alpha^{T_1}$. Since it is not in $\bar{L}_\alpha^{T_1}$
we must have that $I^{T_2} (x) >\alpha$. Therefore it must be the
limit of a sequence of functions in $\bar{L}_\alpha^{T_2}$,
otherwise it could be removed and the set would still be compact
and have the required property. This cannot be the case, as the
functions truncated to the interval $[0,T_1)$ must also tend to
the limit $x$ truncated to $[0,T_1)$. This shows that $x \in
\bar{L}_\alpha^{T_1}$, the required contradiction. Hence there is no $x$ in $\bar{L}_\alpha^{T_2}$ but not in
$\bar{L}_\alpha^{T_1}$.

Define
$L_\alpha=  \bigcap_{T\in \mathbb{N}} \bar{L}_\alpha^T$, which is
compact in the projection topology and has the following property:
\begin{displaymath}
\limsup_{N\rightarrow \infty} \frac{1}{f(N)} \log \mathbb{P}(W^N \notin L_\alpha)<-\alpha\,.
\end{displaymath}
By assumption there exist $t_i$ such that
\begin{displaymath}
\limsup_{N \rightarrow \infty} \frac{1}{f(N)} \log
\mathbb{P}\left(\sup_{t> t_i} [W^N(0,t)-(\lambda+\epsilon_i)t]\geq 1\right)
\leq -\alpha\,,
\end{displaymath}
where $\epsilon_i= 1/ i$. Let $t_i^+$ be the minimum $t$
such that this holds. Also, let $t_i^-$ be such that
\begin{displaymath}
\limsup_{N \rightarrow \infty} \frac{1}{f(N)} \log
\mathbb{P}\left(\sup_{t> t_i^-} [W^N(0,t)-(\lambda-\epsilon_i)t]\leq 1\right)
\leq -\alpha\,.
\end{displaymath}
We define $K_\alpha$ to be the set of continuous functions which have the following property:
$$ x(t)< (\lambda + 1/i)t +1 \textrm{  for $t_i^+<t<t_{i+1}^+$}\,, $$
and
$$ x(t)> (\lambda - 1/i)t -1 \textrm{  for $t_i^-<t<t_{i+1}^-$}\,. $$
Obviously we then have that
$$\limsup_{N\rightarrow \infty} \frac{1}{f(N)} \log \mathbb{P}(W^N\notin K_\alpha)<-\alpha\,.$$
We now define $M_\alpha = K_\alpha \cap L_\alpha$, which has the property that
$$\lim_{\alpha \rightarrow \infty} \limsup_{N \rightarrow \infty} \frac{1}{f(N)} \log \mathbb{P}(W^N \notin M_\alpha) =-\infty\,.$$
Also, we know that $L_\alpha$ is compact with the projection topology. Hence, given a sequence $x^{(n)}$ in  $M_\alpha$, we can find a subsequence $x^{(j)}$ which converges to some $x$ in the projective topology. In order to complete the proof we need to show $x \in M_\alpha$ and then that $x^{(j)} \rightarrow x$ in the scaled uniform topology.

Since $x^{(j)} \rightarrow x$ uniformly on compact intervals,
$$
\lim_{j \rightarrow \infty} \sup_{t\in [0,T]} \left |
\frac{x^{(j)}(t)}{1+t} - \frac{x(t)}{1+t} \right | =0 \textrm{  for
every $T>0$}\,.
$$ 
Also, since $x^{(j)} \in K_\alpha$, we have $$\left |
\frac{x^{(j)}(t)}{t}-\lambda \right | \leq \frac{1}{t}
+\epsilon_t \textrm{ for all }t> \max( t_1^-, t_1^+)\,,$$ 
where
$\epsilon_t = 1/\min\{ i-1; t_i^+>t, t_i^->t\}$, which tends to 0 as
$t\rightarrow \infty$.  This shows that $$\left |
\frac{x(t)}{t}-\lambda \right | \leq \frac{1}{t}
+\epsilon_t \textrm{ for all }t> \max( t_1^-, t_1^+)\,,$$ 
thus giving
$x\in M_\alpha$. Finally, given $\epsilon>0$ we choose $i$ such
that $1/i <\epsilon$ and then select $T>0$ such that $T>t_i^+$ and
$T>t_i^-$. Then, for $j$ sufficiently large, we have
\begin{align*}
\| x^{(j)}-x \|  \leq\sup_{t\leq T} \left | \frac{x^{(j)}(t)}{1+t} - \frac{x(t)}{1+t} \right | +\sup_{t>T} \left | \frac{x^{(j)}(t)}{1+t} - \frac{x(t)}{1+t} \right |\leq 2\epsilon.
 \end{align*}
\end{proof}

We now make use of these lemmas to prove the large deviations principles for the different scalings of interest. To do this we need to carry out the following three tasks in each setting:
\begin{enumerate}
\item Calculate $\Omega_{\alpha, \beta}(\theta, j,T)$ and check the necessary conditions.
\item Calculate the corresponding rate function.
\item Check that $\tilde{\bar{A}}_{\alpha, \beta}$ restricted to the interval $[0,T)$ is exponentially tight in the space $\mathscr{C}^T$ with the uniform norm.
\end{enumerate}

\subsection{The case \texorpdfstring{$\alpha=\beta=1$}{}}

In the case $\alpha=\beta=1$ our proof is relatively brief, as nearly all of the above three points are dealt with directly by the assumptions placed on the process. 

\begin{pfof}{\textbf{Theorem \ref{thm::ldlb}}.}

We start by examining 
$$\Omega_{1,1}(\theta, j,T)= \log \mathbb{E} \left( e^{\sum_{i=1}^{\lvert j \rvert}  \theta_i A(j_{i-1} T,  j_{i} T )}\right) - \lambda\mathbb{E}(Y)\sum_{i=1}^{\lvert j \rvert}\theta_i T(j_i -j_{i-1}).$$ 
This is the finite distributional log moment generating function for the process $A$, minus a linear function of $\theta$. As long as the log moment generating function satisfies the necessary conditions, $\Omega_{1,1}$ will also obey them. As it is a log moment generating function we automatically have that $\Omega_{1,1}(0, j,T) =0$ and that it is differentiable on the finite domain. Then Assumption \ref{ass::ldlb}.3 gives the necessary steepness condition. 

By Lemma \ref{lem::finiteA}, the finite distributions of $\tilde{A}^N_{1,1}$ and $\tilde{\bar{A}}^N_{1,1}$ obey large deviations principles.
Now, we know by Assumption \ref{ass::ldlb}.4 that $\tilde{\bar{A}}^N_{1,1}$ is  exponentially tight in $\mathscr{C}^T$. By Lemma \ref{lem::finitesp} we get the sample path result for finite time which we can extend using Lemma \ref{lem::infinitesp} in conjunction with Lemma \ref{lem::lt2} to a sample path large deviations principle on $\mathscr{C}$ with the topology induced by the scaled uniform norm $||\cdot||_s$. The issue is then showing that we can restrict ourselves to $\mathscr{C}_{0}$; this is done with Lemma \ref{lem::C0} and Lemma 4.1.5 in \cite{Dembo1998}. Then, as $\mathscr{C}_0$ is a subspace of $\mathcal{D}_0$, we can expand the space to this. Finally, the exponential tightness of $\tilde{\bar{A}}_{1,1}^N$ and $\tilde{A}^N_{1,1}$ gives the result.
\end{pfof}

\subsection{The case \texorpdfstring{$0< \alpha <\beta =1$}{}}

Consider the case $0< \alpha <\beta =1$. In this section we outline the proof of our Theorem \ref{thm::ldsb}, which has previously been established by Cruise \cite{Cruise2009a}, within our framework.  We use the following two lemmas.  The first gives the log-moment generating function and is proved analogously to \ref{lem::ratemsb} below. The second establishes exponential tightness of the scaled process, and may be proved similarly to Lemma \ref{lem::exptmsb}.

\begin{lemma}\label{lem::ratesb}
For any $0<\alpha <1$, we have that $$ \Omega_{\alpha, 1}(\theta, j,T)=T \sum_{i=1}^{\lvert j \rvert}  (\lambda (M(\theta_i)-1) -\lambda\mathbb{E}(Y)\theta_i )(j_{i}-j_{i-1})\,.$$ Furthermore, we have
$$I_{\alpha,1}(x)= \begin{cases}
\int_0^{\infty} \Omega^*(\dot{x}+\lambda \mathbb{E}(Y)) dt &\textrm{ if $x(0)=0$ and $x\in\mathscr{AC}$\.,} \\
\infty &\textrm{otherwise\,,}
\end{cases}
$$
where $\Omega^*(y)=\sup_{\theta \in \mathbb{R}} \big[ \theta y - \lambda(M(\theta)-1) \big]$.
\end{lemma}

\begin{lemma}\label{lem::exptsb}
For $A$ which obeys Assumptions \ref{ass::ldsb} and with $0<\alpha<1$, we have that $\tilde{\bar{A}}^N_{\alpha,1}$ is exponentially tight on $\mathscr{C}^T$ with the uniform topology.
\end{lemma}

Using these lemmas we have the large deviations result.
\begin{pfof}{\textbf{Theorem \ref{thm::ldsb}}.}

Firstly, Lemma \ref{lem::finiteA} show us that we have large deviations principles for the finite distributions of $\tilde{A}_{\alpha,1}^N$ and $\tilde{\bar{A}}_{\alpha,1}^N$. Now we extend this to a sample path result on the interval $[0,T)$ for any $0<T<\infty$ for $\tilde{\bar{A}}_{\alpha,1}^N$ in the space $\mathscr{C}^T$ with the uniform topology using Lemmas \ref{lem::finitesp} and \ref{lem::exptsb}. This is then extended to the whole of $\mathscr{C}$ with the scaled uniform norm by Lemma \ref{lem::infinitesp} and \ref{lem::lt2}. Now, for all $N$ we have that $\tilde{\bar{A}}^N_{\alpha,1}$ is in $\mathscr{C}_0$; by Lemma \ref{lem::C0} we can use Lemma 4.1.5 in \cite{Dembo1998} to restrict to this space. We  expand $\mathscr{C}_0$ to $\mathcal{D}_0$ and use exponential tightness to get that $\tilde{A}^N_{\alpha,1}$ obeys a sample path large deviations principle and Lemma \ref{lem::ratesb} demonstrates that we have the required rate function.
\end{pfof}

\subsection{The case \texorpdfstring{$1/2< \alpha= \beta <1$}{}}

The proof for the case $1/2< \alpha= \beta <1$ follows along similar lines to the previous results, with the major difference being that we obtain a neat form for $\Omega_{\alpha,\alpha}(t)$ in this setting, which enables a simplified rate function. As mentioned previously, this has the form of a rate function for a Gaussian process which leads us to using the Generalized Schilder's Theorem as proved in \cite{Addie2002}.

\begin{pfof}{\textbf{Theorem \ref{thm::mdlb}}.}

We start by examining 
\begin{equation}\label{eqn::asdf}
\Omega_{\alpha, \alpha}  (\theta, j, T)= \lim_{N \rightarrow \infty} \frac{1}{N^{2(\alpha-1)}} \log \mathbb{E} \left( e^{ \sum_{i=1}^{|j|} N^{\alpha -1} (\theta_i A(j_{i-1}T, j_i T) -\theta_i T(j_i- j_{i-1}) \lambda \mathbb{E}(Y) } \right)\,.
\end{equation}
We use the Taylor expansion to consider this for large $N$, since $N^{\alpha -1}$ is small. This is possible because, for small $\theta$, $\Lambda_t(\theta) < \infty$  for all $t$. This gives
$$(\ref{eqn::asdf})= \lim_{N \rightarrow \infty} \frac{1}{N^{2(\alpha-1)}}  \left( \sum_{i=1}^{|j|} \sum_{l=1}^{|j|} \Gamma(j_i,j_l) \theta_i \theta_l N^{2(\alpha -1)} + o(N^{3(\alpha-1)}) \right)\,,
$$
where $\Gamma$ is the covariance function. This is a quadratic function in $\theta$, which means that it has the necessary conditions for us to apply Lemma \ref{lem::finiteA} to show that the finite distributions of $\tilde{A}^N_{\alpha,\alpha}$ and $\tilde{\bar{A}}^N_{\alpha,\alpha}$ obey large deviations principles. 

Now, we know by Assumption \ref{ass::mdlb}.4 that $\tilde{\bar{A}}^N_{\alpha,\alpha}$ is  exponentially tight in $\mathscr{C}^T$. So we now apply Lemma \ref{lem::finitesp} to get the sample path result for finite time, which we can extend using Lemma \ref{lem::infinitesp} in conjunction with Lemma \ref{lem::lt3} to a sample path large deviations principle on $\mathscr{C}$ with the topology induced by the scaled uniform norm $||\cdot||_s$. The issue is then showing that we can restrict ourselves to $\mathscr{C}_{0}$; this is done with Lemma \ref{lem::C0} and Lemma 4.1.5 in \cite{Dembo1998}. Then, as $\mathscr{C}_0$ is a subspace of $\mathcal{D}_0$, we can expand the space to this. Finally, the exponential tightness of $\tilde{\bar{A}}_{\alpha,\alpha}^N$ and $\tilde{A}^N_{\alpha,\alpha}$ leaves us only the simplified rate function to establish.

We consider the behaviour of the variance function, as this governs the behaviour of $\Gamma$. We note that, since $\Psi^2_{\infty, d\rightarrow t}>0$ and $\Psi^2_{\infty, d\rightarrow t}>0$, we have that $v(t)/t^2 \rightarrow 0$ as $t\rightarrow \infty$ and the variance function is continuous.
We consider a Gaussian process $Z$ with variance function $v$ and mean $0$, which we use to find the rate function.  $Z/N^{1/2}$ obeys a large deviations principle by the Generalized Schilder's Theorem \cite[Theorem 5.2.3]{Dembo1998}. We consider a projection of the process indexed by $T>0$ and $j$ a finite partition of $[0,1]$ given by the vector
$$(Z(Tj_1)-Z(0), Z(Tj_2)-Z(Tj_1),\dots, Z(T)-Z(T j_{|j|-1}))/N^{1/2}\,.$$
This then obeys a large deviations principle with rate $$I_{j,T}(x) = \sup_{\theta} (\langle x,\theta\rangle -\Omega_{\alpha, \alpha}) (\theta,j,T).$$ Since this is the same rate function as for the finite distributions of $\tilde{A}_{\alpha,\alpha}^N$, the uniqueness of rate functions give the rate function for $\tilde{A}_{\alpha,\alpha}^N$.
\end{pfof}

\subsection{The cases \texorpdfstring{$\alpha<\beta <1$}{}, \texorpdfstring{$\alpha + \beta >1$}{} and \texorpdfstring{$0<\alpha< 1$}{}, \texorpdfstring{$\beta >1$}{}}

Finally, we consider the cases (iv) and (v): $\alpha<\beta <1$, $\alpha + \beta >1$ and $0<\alpha< 1$, $\beta >1$.  Here the proofs again follow a similar style to previous cases. We will state the lemmas as for the previous cases, but omit later proofs for brevity; Lemmas 
\ref{lem::ratebg1} and \ref{lem::exptbg1}
may be obtained by analogous arguments to those used for Lemmas \ref{lem::ratemsb}  and \ref{lem::exptmsb}.  We will, however, give some calculations to motivate the form of the rate function in the lightly loaded case, $\beta>1$. 

\begin{lemma}\label{lem::ratemsb}
For a given $0<\alpha<\beta <1$ and $\alpha + \beta >1$,  we have that $$\label{eq:514} \Omega_{\alpha, \beta}(\theta, j,T)=T \sum_{i=1}^{\lvert j \rvert}  \lambda \mathbb{E}(Y^2)\theta_i^2/2\,.$$ Furthermore,
$$I_{\alpha,\beta}(x)= \begin{cases}
\int_0^{\infty} \frac{\dot{x}^2}{2 \lambda \mathbb{E}(Y^2)} dt &\textrm{ if $x(0)=0$ and $x\in\mathscr{AC}$\,,} \\
\infty &\textrm{otherwise\,.} 
\end{cases}
$$
\end{lemma}
\begin{proof}
We have
\[
\Omega_{\alpha, \beta}(\theta, j,T) = \lim_{N\rightarrow \infty}  N^{2-\alpha-\beta} \log \mathbb{E} \left( e^{\sum_{i=1}^{\lvert j \rvert}  N^{\beta-1}\theta_i A( N^{\alpha -\beta} j_{i-1} T, N^{\alpha -\beta} j_{i} T )
 - \theta_i T(j_i -j_{i-1})N^{\alpha -1}\lambda\mathbb{E}(Y)}\right).
\]
Now we let $\Delta = T N^{\alpha -\beta}$, which means $\Delta \rightarrow 0$ as $N \rightarrow \infty$.  Let $f(\Delta)$ be such that $f(\Delta)=N^{\beta-1}$, so that
\[
\Omega_{\alpha, \beta}(\theta, j,T)\\ = \lim_{\Delta \rightarrow 0} \frac{T}{\Delta f(\Delta)^2} \log \mathbb{E} \left( e^{\sum_{i=1}^{\lvert j \rvert}  f(\Delta)\theta_i A( \Delta j_{i-1}, \Delta j_{i} ) - f(\Delta)\theta_i \Delta (j_i -j_{i-1})\lambda\mathbb{E}(Y)}\right)\,.
\] 
We can then apply Lemma \ref{lem::pple3}, yielding
$$\Omega_{\alpha, \beta}(\theta, j,T)= \frac{1}{2}T\lambda\mathbb{E}(Y^2) \sum_{i=1}^{\lvert j \rvert}  \theta_i^2(j_{i}-j_{i-1})\,.$$
So, by Lemma \ref{lem::finiteA} the rate function is
\begin{align*}I_{j,T}(\mathbf{z})&=
\sup_{\theta \in \mathbb{R}^{\lvert j \rvert} } \left\{ \sum_{i=1}^{\lvert j \rvert} (z_i-z_{i-1}) \theta_i - \frac{1}{2}T\lambda\mathbb{E}(Y^2) \sum_{i=1}^{\lvert j \rvert}  \theta_i^2(j_{i}-j_{i-1}) \right \} \\
&= \sup_{\theta \in \mathbb{R}^{\lvert j \rvert} } \left\{ \sum_{i=1}^{\lvert j \rvert} T(j_i-j_{i-1})\left( \left(\frac{(z_i-z_{i-1})}{T(j_i-j_{i-1})} \right) \theta_i - \frac{1}{2}\lambda\mathbb{E}(Y^2) \theta_i^2 \right) \right\} \\
&=   \sum_{i=1}^{\lvert j \rvert} T(j_i-j_{i-1}) \Omega^*\left(\frac{(z_i-z_{i-1})}{T(j_i-j_{i-1})} \right)\,,
\end{align*}
where $$\Omega^*(y)=\sup_{\theta \in \mathbb{R}} \left(\theta y - \frac{1}{2} \lambda\mathbb{E}(Y^2) \theta^2\right)= \frac{y^2}{2\lambda\mathbb{E}(Y^2)}. $$
So, we have that
\begin{align*}
I_T(x) &= \sup_{j\in \mathscr{J}, 0<T'\leq T} I_{j,T'}(x) \\
& =\sup_{j\in \mathscr{J}, 0<T'\leq T} \sum_{i=1}^{\lvert j \rvert} T'(j_i-j_{i-1}) \Omega^*\left(\frac{(x(T'j_i)-x(T'j_{i-1}))}{T'(j_i-j_{i-1})} \right)\,.
\end{align*}
We want to show this is equal to 
$$I^T(x)= \begin{cases}
\int_0^{T} \frac{\dot{x}^2}{2 \lambda \mathbb{E}(Y^2)} dt &\textrm{ if $x(0)=0$ and $x\in\mathscr{AC}$\,,} \\
\infty &\textrm{otherwise\,.} 
\end{cases}
$$
Firstly,  $\Omega^*$ is non-negative for all $y$, as the value at $\theta =0$ is $0$, so the supremum  is at $j_{|j|}=1$ and $T'=T$.
In addition,  we have that $I_T \leq I^T$ by 
Jensen's inequality and convexity of $\Omega^*$ 
\cite[Lemma 2.2.5]{Dembo1998}. We need to prove the reverse inequality. Let us
start with  $f$, which is absolutely continuous, and let $g(t)
= d f /dt$, which is in $L_1([0,T])$. Then, for $k>1$ we can
define $\tau_k = T/k$ and
\begin{align*}
g^{(k)}(t)& = \tau_k^{-1} \int_{\lfloor t/\tau_k\rfloor \tau_k}^{(\lfloor t/\tau_k\rfloor +1)\tau_k} g(s) ds \textrm{  for $t\in[0,T)$}\,, \\
g^{(k)} (T) &= \tau_k^{-1} \int_{T-\tau_k}^{T} g(s) ds\,.
\end{align*}
Using these definitions we have
\[
I_T (f) \geq \liminf_{k\rightarrow \infty} \sum_{l=1}^{k} \tau_k \Omega^* \left[ \frac {f\left(l \tau_k \right)-f \left((l-1) \tau_k\right) } {\tau_k}  \right] 
=\liminf_{k\rightarrow \infty} \int_0^{T} \Omega^*(g^{(k)}(t))dt\,.
\]
In addition to this, by Lebesgue's theorem, $\lim_{k\rightarrow \infty} g^{(k)}(t) =g(t)$ almost everywhere in $[0,T]$. So, by Fatou's lemma and the lower semicontinuity of $\Omega^*$, we have
\[
\liminf_{k\rightarrow \infty} \int_0^{T} \Omega^*(g^{(k)}(t))dt \geq \int_0^{T} \liminf_{k\rightarrow \infty}  \Omega^*(g^{(k)}(t))dt 
\geq  \int_0^{T} \Omega^*(g(t))dt = I^T(f),
\]
which together show that  $I_T(f) \geq I^T(f)$, for
$f$ absolutely continuous.

Now let $f$ be a function from $[0,T]$ to $\mathbb{R}$ which is
not absolutely continuous. This implies that there exists
$\delta>0$ and $s_1^n<t_1^n \leq \cdots \leq s_{k_n}^n<t_{k_n}^n$ such that $\sum_{l=1}^{k_n} (t_l^n -s_l^n) \rightarrow 0$ but
$\sum_{l=1}^{k_n} |f(t_l^n) -f(s_l^n)| \geq \delta$.  In addition,
 $\Omega^*$ is non-negative, so we get
\begin{align*}
I_T(f) &= \sup_{\substack{0<t_1<t_2<\cdots <t_k =T\\ \theta_1 , \ldots , \theta_k \in \mathbb{R}}} \sum_{l=1}^{k} \theta_l(f(t_l)-f(t_{l-1})) -(t_l-t_{l-1})\frac{1}{2} \lambda\mathbb{E}(Y^2) \theta_l^2 \\
&\geq \sup_{\substack{0\leq s_1<t_1\leq s_2 <t_2<\cdots \leq s_k <t_k =T\\ \theta_1 , \ldots , \theta_k \in \mathbb{R}}} \sum_{l=1}^{k} \theta_l(f(t_l)-f(s_l)) -(t_l-s_l) \frac{1}{2} \lambda\mathbb{E}(Y^2) \theta_l^2\,.
\end{align*}
Now let $t_l=t_l^n$, $s_l=s_l^n$ and $\theta_l$ have the same sign as $f(t_l^n)-f(s_l^n)$, with $|\theta_i|=\rho$. This then gives
\[
I_T(f) \geq \limsup_{n \rightarrow \infty} \left\{ \rho \sum_{l=1}^{k_n} |f(t_l^n) -f(s_l^n)|  - \sup_{|\theta|=\rho} \left\{ \frac{1}{2} \lambda\mathbb{E}(Y^2) \theta^2 \right\} \sum_{l=1}^{k_n} (t_l^n -s_l^n) \right\} \geq \rho \delta .
\]
The choice of $\rho$
is arbitrary, implying that $I_T(f)=\infty$.

Finally, we have that $I_T(x)$ is increasing in $T$ as $\Omega^*$ is non-negative. So, $$I_{\alpha,\beta}(x) =\lim_{T\rightarrow \infty} I_T(x) =  \begin{cases}
\int_0^{\infty} \frac{\dot{x}^2}{2 \lambda \mathbb{E}(Y^2)} dt &\textrm{ if $x(0)=0$ and $x\in\mathscr{AC}$\,,} \\
\infty &\textrm{otherwise\,.} 
\end{cases}$$
\end{proof}
\begin{lemma}\label{lem::exptmsb}
Let $0<\alpha<\beta <1$ and $\alpha + \beta >1$.  For $A$ which satisfies Assumptions \ref{ass::mdsb} we have that $\tilde{\bar{A}}^N_{\alpha,\beta}$ is exponentially tight on $\mathscr{C}^T$ with the uniform topology.
\end{lemma}
\begin{proof}
For a fixed $\gamma>0$, let
\begin{displaymath}
K_\gamma (T) = \left \{ f \in \mathscr{AC} : f(0)=0, \int_0^{\infty} \frac{\dot{x}^2}{2 \lambda \mathbb{E}(Y^2)} \leq \gamma + 1 \right\}\,.
\end{displaymath}
We want to show that $\bar{K}_{\gamma} (T)$ is compact and that $$\lim_{\gamma \rightarrow \infty} \limsup_{N \rightarrow
\infty} \frac{1}{N^{\alpha}} \log \mathbb{P}(\tilde{\bar{A}}^N_{\alpha,\beta}
|_{T} \notin K_\gamma (T)) = - \infty.$$ We show the second of these
first. We have that $\mathbb{P}( \tilde{\bar{A}}^N_{\alpha,\beta}
|_{T} \notin K_\gamma(T)) =\mathbb{P}( \tilde{\bar{A}}^N_{\alpha,\beta}
|_{T} \notin K_\gamma(T) \cap \mathscr{AC})$,
as, by definition, paths of $\tilde{\bar{A}}^N_{\alpha,\beta}
|_{T}$ are almost surely absolutely continuous. We also have that $\tilde{\bar{A}}^N_{\alpha,\beta}
|_{T}$ obeys a large deviations principle in this
space, but with the topology of pointwise convergence by the first half of Lemma \ref{lem::finitesp}. So,
\begin{displaymath}
\limsup_{N \rightarrow \infty} \frac{1}{N^{\alpha}} \log \mathbb{P}(\tilde{\bar{A}}^N_{\alpha,\beta}
|_{T} \notin K_\gamma(T)) \leq -\inf_{f \in \overline{K^c_\gamma(T)\cap \mathscr{AC}}} I_T(f)\,,
\end{displaymath}
where $K^c_\gamma(T)$ is the complement of $K_\gamma(T)$ and $\bar{K}$ is the closure of the set $K$.
By the definition of $K_\gamma (T)$ we have $I_T(f) > \gamma$ for $f \in \overline{K^c_\gamma(T)\cap \mathscr{AC}}$. If there exists $g \in \overline{K^c_\gamma(T)\cap
\mathscr{AC}}$ such that $I_T(g) \leq \gamma$, then there exists a sequence $f_n \in K^c_\gamma(T)\cap
\mathscr{AC} $  whose limit is $g$; but this cannot be the case because there would have to exist $N$ such that $I_T(f_N)< \alpha +
\beta$, and this is untrue by the definition of
$K_\gamma(T)$. Hence,
\begin{displaymath}
\limsup_{N \rightarrow \infty} \frac{1}{N^{\alpha}} \log \mathbb{P}(\tilde{\bar{A}}^N_{\alpha,\beta}
|_{T} \notin K_\gamma (T)) \leq -\gamma\,.
\end{displaymath}
This goes to $-\infty$ as $\gamma \rightarrow \infty$.

To prove compactness we make use of the Arzel\'a-Ascoli theorem,
which says that if $K_\gamma (T)$ is a closed and
bounded set of equicontinuous functions then it is compact.  To check the functions are equicontinuous, if $f\in K_\gamma (T)$ then the
continuous function $f$ is differentiable almost everywhere on
$[0,T]$, and, for all $0\leq s \leq t\leq T$,
\begin{displaymath}
\Omega^* \left( \left(\frac{f(t)-f(s)}{t-s}\right)^2 \cdot \frac{1}{2\lambda\mathbb{E}(Y)}\right) \leq  \frac{1}{t-s} \int_s^t \Omega^*\left(\frac{\dot{f}(t')}{2\lambda\mathbb{E}(Y)}\right) dt' \leq \frac{\gamma+1}{t-s}\,.
\end{displaymath}
By definition, the functions in the set $K_\gamma (T)$ have bounded derivatives. As such, in any interval of time $\delta$, the variation of the function is bounded. In addition, we have a bound on the set by setting $s=0$ and
$\delta =T$.
\end{proof}
\begin{lemma}\label{lem::ratebg1}
For a given $0<\alpha<1$ and $\beta >1$,  we have that $$ \Omega_{\alpha, \beta}(\theta, j,T)= \left \lbrace \begin{array}{cc}
    0 & \text{ if } \theta_j < \beta - 1 \,  \forall j\,, \\
    1 &  \text{ if } \theta_j = \beta-1 \, \forall j\,,\\
    \infty & \text{ otherwise\,, } 
\end{array} \right. 
$$
and
\begin{align*}
    I_{\alpha, \beta}(x) = \begin{cases}
\int_0^{\beta - 1} \dot{x} dt &\textrm{ if $x(0)=0$ and $x\in\mathscr{AC}$\,,} \\
\infty &\textrm{otherwise\,.} 
\end{cases}
\end{align*}
\end{lemma}
To see why $\Omega_{\alpha, \beta}(\theta, j,T)$ takes this form in Lemma \ref{lem::ratebg1}, note that, for $0<\alpha<1$ and $\beta > 1$,
\begin{align*}
    \Omega_{\alpha, \beta}(\theta, j, T) = \lim_{N \to \infty} \frac{N^{1-\alpha}}{\log N} \log \mathbb{E}\left[e^{\langle\theta \log N, A^{j, N^{\alpha-\beta}T}\rangle} \right]\,.
\end{align*}
By replacing $t$ with $N^{\alpha - \beta}t$ in Lemma \ref{lem::pple}, we can see that the process $A(N^{\alpha - \beta}t_{j-1}, N^{\alpha-\beta}t_j)$ can be approximated by a Poisson process with mean $\lambda N^{\alpha - \beta}(t_j - t_{j-1})$.  Therefore,
\begin{align*}
    \Omega_{\alpha, \beta}(\theta, j, T) & = \lim_{N \to \infty} \frac{N^{1-\alpha}}{\log N} \prod_j \mathbb{E}\left[e^{\theta_j \log N A(N^{\alpha -\beta}t_{j-1}, N^{\alpha - \beta}t_j)}\right]\\
            & =  \lim_{N \to \infty} \frac{N^{1-\alpha}}{\log N} \sum_{j} N^{\alpha - \beta} \lambda (t_j - t_{j-1}) (e^{\theta_j \log N} -1) \\
            & =  \lim_{N \to \infty} \frac{N^{1 - \beta}}{\log N} \sum_{j} \lambda (t_j - t_{j-1})(N^{\theta_j} - 1) \\
            & = \lim_{N \to \infty} \sum_j \lambda (t_j - t_{j-1}) \left( \frac{N^{1 - \beta + \theta_j}}{\log N} \right) \\
            & = \begin{cases}
                    0 &\textrm{ if } \theta_j < \beta-1 \quad \forall j\,, \\
                    1 & \textrm{ if } \theta_j = \beta-1 \quad \forall j\,, \\
                    \infty & \textrm{ otherwise\,. }
                \end{cases}
\end{align*}
\begin{lemma}\label{lem::exptbg1}
Let $0<\alpha<1$ and $\beta >1$.  For $A$ which obeys Assumptions \ref{ass::ldbg1} we have that $\tilde{\bar{A}}^N_{\alpha,\beta}$ is exponentially tight on $\mathscr{C}^T$ with the uniform topology.
\end{lemma}

Using these lemmas we can now prove the large deviations results.
\begin{pfof}{\textbf{Theorems \ref{thm::mdsb} and \ref{thm::ldbg1}.}}
We begin with Theorem \ref{thm::mdsb}. Lemma \ref{lem::finiteA}  show us that we have large deviations principles for the finite distributions of $\tilde{A}^N_{\alpha,\beta}$ and $\tilde{\bar{A}}^N_{\alpha,\beta}$. Now we extend this to a sample path result on the interval $[0,T)$ for any $0<T<\infty$ for $\tilde{\bar{A}}^N_{\alpha,\beta}$ in the space $\mathscr{C}^T$ with the uniform topology using Lemmas \ref{lem::finitesp} and \ref{lem::exptmsb}. This is then extended to the whole of $\mathscr{C}$ with the scaled uniform norm by Lemma \ref{lem::infinitesp} and \ref{lem::lt3}. For all $N$,  $\tilde{\bar{A}}^N_{\alpha,\beta}$ belongs to $\mathscr{C}_0$ (Lemma \ref{lem::C0}). We use Lemma 4.1.5 in \cite{Dembo1998} to restrict to this space. We expand $\mathscr{C}_0$ to $\mathcal{D}_0$ and use exponential tightness to get that $\tilde{A}^N_{\alpha,\beta}$ obeys a sample path large deviations principle, and Lemma \ref{lem::ratemsb} demonstrates that we have the required rate function.  This  established Theorem \ref{thm::mdsb}.  

The proof of Theorem \ref{thm::ldbg1} is similar, with Lemmas \ref{lemma: time-bounding for beta > 1}, \ref{lem::ratebg1} and \ref{lem::exptbg1} used in place of Lemmas \ref{lem::lt3},\ref{lem::ratemsb} and \ref{lem::exptmsb} respectively.
\end{pfof}

\bibliographystyle{unsrt}
\bibliography{references}

\begin{thebibliography}{10}

\bibitem{Cruise2009a}
R.~J.~R. Cruise.
\newblock Poisson convergence, in large deviations, for the superposition of
  independent point processes.
\newblock {\em Annals of Operations Research}, 170(1):79--94, 2009.

\bibitem{Botvich1995}
D.~D. Botvich and N.~G. Duffield.
\newblock Large deviations, the shape of the loss curve, and economies of scale
  in large multiplexers.
\newblock {\em Queueing Systems}, 20(3-4):293--320, 1995.

\bibitem{Weiss1986}
A.~Weiss.
\newblock A new technique for analyzing large traffic systems.
\newblock {\em Advances in Applied Probability}, 18(2):506--532, 1986.

\bibitem{doi:10.1080/15326340500481762}
M.~Mandjes.
\newblock Large deviations for complex buffer architectures: The short-range
  dependent case.
\newblock {\em Stochastic Models}, 22(1):99--128, 2006.

\bibitem{Subramanian2011}
V.~G. Subramanian, T.~Javidi, and S.~Kittipiyakul.
\newblock Many-sources large deviations for max-weight scheduling.
\newblock {\em IEEE Transactions on Information Theory}, 57(4):2151--2168,
  2011.

\bibitem{FERNANDEZVEIGA20031376}
M.~Fern\'{a}ndez-Veiga, C.~L\'{o}pez-Garc\'{i}a, J.~C. L\'{o}pez-Ardao,
  A.~Su\'{a}rez-Gonz\'{a}lez, and M.~E. Sousa-Vieira.
\newblock On the effectiveness of the many-sources asymptotic for admission
  control.
\newblock {\em Computer Communications}, 26(12):1376--1391, 2003.

\bibitem{Buffet1994}
E.~Buffet and N.~G. Duffield.
\newblock Exponential upper bounds via martingales for multiplexers with
  {M}arkovian arrivals.
\newblock {\em Journal of Applied Probability}, 31(4):1049--1060, 1994.

\bibitem{Duffield1994}
N.~G. Duffield.
\newblock Exponential bounds for queues with {M}arkovian arrivals.
\newblock {\em Queueing Systems}, 17(3-4):413--430, 1994.

\bibitem{Simonian1995}
A.~Simonian and J.~Guibert.
\newblock Large deviations approximation for fluid queues fed by a large number
  of on/off sources.
\newblock {\em IEEE Journal on Selected Areas in Communications},
  13(6):1017--1027, 1995.

\bibitem{Courcoubetis1996}
C.~Courcoubetis and R.~Weber.
\newblock Buffer overflow asymptotics for a buffer handling many traffic
  sources.
\newblock {\em Journal of Applied Probability}, 33(03):886--903, 1996.

\bibitem{Likhanov1999}
N.~Likhanov and R.~R. Mazumdar.
\newblock Cell loss asymptotics for buffers fed with a large number of
  independent stationary sources.
\newblock {\em Journal of Applied Probability}, 36(1):86--96, 1999.

\bibitem{Shroff2004}
D.~Y. Eun and N.~B. Shroff.
\newblock Analyzing a two-stage queueing system with many point process
  arrivals at upstream queue.
\newblock {\em Queueing Systems}, 48(1-2):23--43, 2004.

\bibitem{Zhao2009}
C.~Zhao and X.~Lin.
\newblock On the queue-overflow probabilities of distributed scheduling
  algorithms.
\newblock In {\em Proceedings of the 48th IEEE Conference on Decision and
  Control, 2009 held jointly with the 2009 28th Chinese Control Conference.
  CDC/CCC 2009}, pages 4820--4825. IEEE, 2009.

\bibitem{Delas2002}
S.~Delas, R.~R. Mazumdar, and C.~P. Rosenberg.
\newblock Tail asymptotics for {HOL} priority queues handling a large number of
  independent stationary sources.
\newblock {\em Queueing Systems}, 40(2):183--204, 2002.

\bibitem{Yang2006}
C.-W. Yang, A.~Wierman, S.~Shakkottai, and M.~Harchol-Balter.
\newblock Tail asymptotics for policies favoring short jobs in a many-flows
  regime.
\newblock {\em ACM SIGMETRICS Performance Evaluation Review}, 34(1):97--108,
  2006.

\bibitem{Yang2012}
C.~Yang, A.~Wierman, S.~Shakkottai, and M.~Harchol-Balter.
\newblock Many flows asymptotics for {SMART} scheduling policies.
\newblock {\em IEEE Transactions on Automatic Control}, 57(2):376--391, 2012.

\bibitem{Yang2006a}
C.-W. Yang and S.~Shakkottai.
\newblock Asymptotic evaluation of delay in the {SRPT} scheduler.
\newblock {\em IEEE Transactions on Automatic Control}, 51(11):1848--1854,
  2006.

\bibitem{Shakkottai2001}
S.~Shakkottai and R.~Srikant.
\newblock Many-sources delay asymptotics with applications to priority queues.
\newblock {\em Queueing Systems}, 39(2-3):183--200, 2001.

\bibitem{Bhadra2010}
S.~Bhadra and S.~Shakkottai.
\newblock Buffer asymptotics for coding over networks.
\newblock {\em IEEE Transactions on Information Theory}, 56(12):6159--6181,
  2010.

\bibitem{Debicki2003}
K.~D{\c{e}}bicki and M.~Mandjes.
\newblock Exact overflow asymptotics for queues with many {G}aussian inputs.
\newblock {\em Journal of Applied Probability}, 40(3):704--720, 2003.

\bibitem{Zwart2004}
B.~Zwart, S.~Borst, and M.~Mandjes.
\newblock Exact asymptotics for fluid queues fed by multiple heavy-tailed
  on--off flows.
\newblock {\em The Annals of Applied Probability}, 14(2):903--957, 2004.

\bibitem{Mandjes2001}
M.~Mandjes and J.~H. Kim.
\newblock Large deviations for small buffers: an insensitivity result.
\newblock {\em Queueing Systems}, 37(4):349--362, 2001.

\bibitem{Ozturk2004}
O.~Ozturk, R.~R. Mazumdar, and N.~Likhanov.
\newblock Many sources asymptotics for networks with small buffers.
\newblock {\em Queueing Systems}, 46(1-2):129--147, 2004.

\bibitem{Mandjes2000}
M.~Mandjes and S.~Borst.
\newblock Overflow behavior in queues with many long-tailed inputs.
\newblock {\em Advances in Applied Probability}, 32(4):1150--1167, 2000.

\bibitem{moddev}
D.~Wischik.
\newblock Moderate deviations in queueing theory.
\newblock
  \url{https://www.cl.cam.ac.uk/~djw1005/Research/ucl_research/moddev.pdf},
  2001.
\newblock Preprint.

\bibitem{Puhalskii1999}
A.~A. Puhalskii.
\newblock Moderate deviations for queues in critical loading.
\newblock {\em Queueing Systems}, 31(3-4):359--392, 1999.

\bibitem{Chang1996}
C.-S. Chang, D.~D. Yao, and T.~Zajic.
\newblock Moderate deviations for queues with long-range dependent input.
\newblock In {\em Stochastic Networks}, pages 275--298. Springer, 1996.

\bibitem{CaoRamanan2002}
J.~Cao and K.~Ramanan.
\newblock A {P}oisson limit for buffer overflow probabilities.
\newblock In {\em Proceedings.Twenty-First Annual Joint Conference of the
  {IEEE} Computer and Communications Societies}. {IEEE}, 2002.

\bibitem{Enachescu2006}
M.~Enachescu, Y.~Ganjali, A.~Goel, N.~McKeown, and T.~Roughgarden.
\newblock Routers with very small buffers.
\newblock In {\em INFOCOM}, 2006.

\bibitem{Raina2005}
G.~Raina and D.~Wischik.
\newblock Buffer sizes for large multiplexers: {TCP} queueing theory and
  instability analysis.
\newblock In {\em Next Generation Internet Networks, 2005}, pages 173--180.
  IEEE, 2005.

\bibitem{Gu2007}
Y.~Gu, D.~Towsley, C.~V. Hollot, and H.~Zhang.
\newblock Congestion control for small buffer high speed networks.
\newblock In {\em INFOCOM 2007. 26th IEEE International Conference on Computer
  Communications. IEEE}, pages 1037--1045. IEEE, 2007.

\bibitem{Vishwanath2009}
A.~Vishwanath, V.~Sivaraman, and D.~Ostry.
\newblock How {P}oisson is {TCP} traffic at short time-scales in a small buffer
  core network?
\newblock In {\em Advanced Networks and Telecommunication Systems (ANTS), 2009
  IEEE 3rd International Symposium on}, pages 1--3. IEEE, 2009.

\bibitem{Eun2008}
D.~Y. Eun and X.~Wang.
\newblock Achieving 100\% throughput in {TCP/AQM} under aggressive packet
  marking with small buffer.
\newblock {\em IEEE/ACM Transactions on Networking (TON)}, 16(4):945--956,
  2008.

\bibitem{Wischik2001}
D.~J. Wischik.
\newblock Sample path large deviations for queues with many inputs.
\newblock {\em Ann. Appl. Probab.}, 11(2):379--404, 2001.

\bibitem{Wischik1999}
D.~J. Wischik.
\newblock The output of a switch, or, effective bandwidths for networks.
\newblock {\em Queueing Systems}, 32(4):383--396, Nov 1999.

\bibitem{Dembo1998}
A.~Dembo and O.~Zeitouni.
\newblock {\em Large Deviations Techniques and Applications}.
\newblock Applications of {M}athematics. Springer, 1998.

\bibitem{5409854}
A.~Vishwanath, V.~Sivaraman, and D.~Ostry.
\newblock How poisson is tcp traffic at short time-scales in a small buffer
  core network?
\newblock In {\em 2009 IEEE 3rd International Symposium on Advanced Networks
  and Telecommunication Systems (ANTS)}, pages 1--3, Dec 2009.

\bibitem{Cao2003}
J.~Cao, W.~S. Cleveland, D.~Lin, and D.~X. Sun.
\newblock {\em Internet Traffic Tends Toward Poisson and Independent as the
  Load Increases}, pages 83--109.
\newblock Springer New York, New York, NY, 2003.

\bibitem{Loynes1962}
R.M Loynes.
\newblock The stability of a queue with non-independent inter-arrival and
  service times.
\newblock In {\em Mathematical Proceedings of the Cambridge Philosophical
  Society}, volume~58, pages 497--520. Cambridge University Press, 1962.

\bibitem{Simonian1994}
A.~Simonian and J.~Guibert.
\newblock Large deviations approximation for fluid queues fed by a large number
  of on/off sources.
\newblock In {\em The Fundamental Role of Teletraffic in the Evolution of
  Telecommunications Networks}, pages 1013--1022. Elsevier, 1994.

\bibitem{Daley2008}
D.~J. Daley and D.~Vere-Jones.
\newblock {\em An Introduction to the Theory of Point Processes}.
\newblock Springer New York, 2008.

\bibitem{MR0232428}
D.~W. M{\"u}ller.
\newblock Verteilungs-{I}nvarianzprinzipien f\"ur das starke {G}esetz der
  grossen {Z}ahl.
\newblock {\em Z. Wahrscheinlichkeitstheorie und Verw. Gebiete}, 10:173--192,
  1968.

\bibitem{Ganesh2004}
A.~Ganesh, N.~O'Connell, and D.~Wischik.
\newblock {\em Big Queues}.
\newblock Springer, Berlin, 2004.

\bibitem{Berlinet2004}
A.~Berlinet and C.~Thomas-Agnan.
\newblock {\em Reproducing Kernel Hilbet Spaces in Probability and Statistics}.
\newblock Springer, New York, 2004.

\bibitem{Ganesh2002}
A.~J. Ganesh and N.~O'Connell.
\newblock A large deviation principle with queueing applications.
\newblock {\em Stochastics and Stochastic Reports}, 73(1-2):25--35, 2002.

\bibitem{Addie2002}
R.~Addie, P.~Mannersalo, and I.~Norros.
\newblock Most probable paths and performance formulae for buffers with
  {G}aussian input traffic.
\newblock {\em European Transactions on Telecommunications}, 13(3):183--196,
  2002.

\end{thebibliography}

\end{document}